\documentclass{amsart}

\usepackage[utf8]{inputenc}

\usepackage[authoryear]{natbib}

\usepackage{amsthm}
\usepackage{amssymb}
\usepackage{amsmath}
\usepackage{amsfonts}
\usepackage[all]{xy}
\usepackage{fancyhdr}
\usepackage[CJKbookmarks,bookmarks=true,colorlinks=true,citecolor=black]{hyperref}
\usepackage{amssymb, latexsym}
\usepackage{mathrsfs}
\usepackage{epsfig, subfigure}
\usepackage{enumerate}
\usepackage{graphicx}
\usepackage{bbold}
\usepackage{float}
\usepackage{nomencl}
\makenomenclature
\usepackage[titletoc]{appendix}
\usepackage[margin = 1in]{geometry}
\usepackage{comment}
\usepackage{tikz-cd}
\usepackage{comment}
\usepackage{kotex}
\usepackage{setspace}

\numberwithin{equation}{section}
\newtheorem{definition}{Definition}[section]

\newtheorem{theorem}[definition]{Theorem}
\newtheorem{lemma}[definition]{Lemma}

\newtheorem{proposition}[definition]{Proposition}
\newtheorem{remark}[definition]{Remark}

\begin{document}

\setstretch{1.1}

\title[Decomposition of one-layer neural networks via the infinite sum of RKBS]{Decomposition of one-layer neural networks via the infinite sum of reproducing kernel Banach spaces}

\author{Seungcheol Shin \and Myungjoo Kang }
\address{}
\curraddr{}
\email{}
\pagenumbering{arabic}

\begin{abstract}
In this paper, we define the sum of RKBSs using the characterization theorem of RKBSs and show that the sum of RKBSs is compatible with the direct sum of feature spaces. Moreover, we decompose the integral RKBS into the sum of $p$-norm RKBSs. Finally, we provide applications for the structural understanding of the integral RKBS class.
\end{abstract}
\maketitle

\section{Introduction}
To analyze the performance of neural networks, the hypothesis space represented by (infinite width) neural networks has been studied. Based on the concept of variation spaces  \citep{kurkova2001bounds,mhaskar2004tractability}, \citet{bach2017breaking} defined the $\mathcal{F}_{1}$ spaces as an integral representation of the neural networks using the total variation norm. In subsequent, \citet{weinan2022representation} defined the Barron spaces employing the path norm and showed that the $\mathcal{F}_{1}$ spaces and the Barron spaces can be isometrically isomorphic when using the Rectified Linear Unit (RELU) activation function.

The concept of Reproducing Kernel Banach Spaces (RKBSs) is a generalization of the Reproducing Kernel Hilbert Spaces (RKHSs), similar to how Banach spaces extend Hilbert spaces \citep{zhang2009reproducing}. Relating to neural networks, \citet{bartolucci2023understanding} defined a class of integral RKBSs which are variants of the $\mathcal{F}_{1}$ spaces. They defined a class of integral RKBSs through the characterization theorem of the RKBS introduced by \cite{combettes2018regularized} which describe an RKBS using a feature space and its associated feature map.

In this study, our primary focus is on a class of integral RKBSs. We aim to decompose this function space and identify its fundamental building blocks. Decomposing a function space entails preserving both its algebraic operations and topological properties. Since we are dealing with RKBS, we additionally need to ensure that the decomposition preserves the property that evaluation functionals remain continuous (see Definition \ref{def. RKBS}). Considering the case of RKHS, there exists a sum of RKHSs that naturally extends the space in a canonical manner, resulting in an RKHS \citep{aronszajn1950theory}. Using this approach, we aim to define the (potentially infinite) sum of RKBSs and investigate its relationship with the feature spaces.

The main questions of this paper are the following:
\begin{itemize}
    \item[(1)] Finding a natural definition for the sum of RKBSs that is compatible with the usual direct sum of Banach spaces.
    \item[(2)] How can we decompose a class of integral RKBSs using the sum defined in question (1)?
\end{itemize}

To answer these questions, we define the sum of RKBSs, see Proposition \ref{prop. infinite sum of RKBSs}, and show that the direct sum of the feature (Banach) spaces is compatible with the sum of RKBSs (Proposition \ref{prop. compatibility}). Roughly speaking, it is well-known that the space of the Radon measures can be decompose as the vast $l^1$ direct sum of $L^{1}$ spaces. As an analogue of the fact described above, we decompose the class of integral RKBSs using the sum of p-norm RKBSs (Theorem \ref{thm. decomposition of the integral RKBS}).

In general, when hypothesis spaces satisfy the inclusion $\mathcal{F}_{1} \subset \mathcal{F}_{2}$
(e.g., RKHS), it is expected that $\mathcal{F}_{2}$
exhibits better approximation performance, while 
$\mathcal{F}_{1}$ provides better generalization performance. Therefore, comparing the inclusion relationships among hypothesis spaces plays a crucial role in evaluating model performance. As previously mentioned, taking the sum of hypothesis spaces—particularly RKHSs—serves as a means of expanding the function space, thereby enabling such comparisons. However, to the best of our knowledge, the theoretical foundation for the sum of infinitely many hypothesis spaces remains underdeveloped. This work aims to fill this gap by providing a theoretical basis for such infinite-sum hypothesis spaces.

\subsection{Related Work}
Before the era of neural networks, one of the main topics in machine learning was the kernel method, exemplified by concepts such as Reproducing Kernel Hilbert Spaces (RKHSs) and Support Vector Machines \citep{aronszajn1950theory,steinwart2008support,berlinet2011reproducing}.
Machine learning models such as SVMs, which adopt an RKHS as their hypothesis space, are supported by both existence theorems and the Representer Theorem for general L-risk minimization. In particular, when the empirical L-risk is considered, the Representer Theorem guarantees the form of the solution, and algorithms for explicitly computing such solutions have been well established \citep{smola1998learning,shalev2014understanding}. These characteristics substantially bridge the gap between theoretical understanding and practical application.
A standard approach to extending RKHSs is through their sum. For instance, methods such as the multiple kernel algorithm demonstrate that such extensions are effective in capturing diverse aspects of data \citep{yamanishi2004protein,gonen2011multiple}.

However, RKHS-based learning algorithms exhibit certain limitations due to their inner product structure. To address these challenges, the concept of Reproducing Kernel Banach Spaces (RKBSs) was introduced. Numerous studies have explored its theoretical foundations and applications \citep{zhang2009reproducing,song2013reproducing,fasshauer2015solving,lin2022reproducing}. Meanwhile, early theoretical research on neural networks primarily focused on approximation properties \citep{cybenko1989approximation,hornik1990universal,barron1993universal}. This line of inquiry led to further investigations into the hypothesis spaces of infinitely wide neural networks, culminating in the introduction of concepts such as Barron spaces and variation spaces \citep{bach2017breaking,ma2022barron,weinan2022representation,siegel2023characterization}. Recent studies have attempted to analyze the hypothesis spaces of neural networks within the RKBS framework. For this purpose, the concept of integral RKBS has been introduced, which allows the Representer Theorem to be established for the empirical L-risk in one-layer neural networks \citep{bartolucci2023understanding}.

Various results have been established for representer theorems and regularized optimization problems, particularly in the context of optimizing functionals such as the empirical L-risk, which can be embedded in finite-dimensional spaces. The core idea behind these results involves the appropriate use of the Krein–Milman theorem and Carathéodory’s theorem. In some approaches, duality maps are employed instead of assuming Gâteaux differentiability \citep{boyer2019representer,bredies2020sparsity,unser2021unifying,unser2022convex}. For more general L-risk functionals, methods have been developed that characterize the solution by analyzing the subdifferential of the risk and showing that the norm on the hypothesis space—up to reflexive Banach spaces—is Gâteaux differentiable \citep{steinwart2008support,combettes2018regularized}. However, rather than proving a representer theorem directly, this paper focuses on how the known form of the solution—derived from existing representer theorems—can be reformulated through our infinite-sum decomposition approach. Additionally, for the completeness of the paper, we also establish a general existence result for one-layer neural networks under general L-risk.

\subsection{Organization}
This paper is organized as follows. In Section 2, we briefly review the definitions and basic facts of the functional analysis. In Section 3, following \citet{bartolucci2023understanding,spek2022duality}, we introduce the definition of RKBSs and related function subclasses, namely a class of integral RKBSs and a class of p-norm RKBSs. We present some basic properties of these function classes, particularly focusing on the comparison between integral RKBSs and spaces of continuous functions (Proposition \ref{prop. A is compact operator}). 
Moreover, we define the sum of RKBSs, which is a modified version of Example 3.13 in \citet{combettes2018regularized} and the theorem in 353p of \citet{aronszajn1950theory}, by using the characterization theorem of an RKBS. 
In Section 4, we state the main result of this article. We provide the compatibility between the sum of RKBSs and the direct sum of feature (Banach) spaces. Furthermore, using the compatibility (Proposition \ref{prop. compatibility}), we obtain that a class of integral RKBSs can be decomposed into the sum of p-norm RKBSs (Theorem \ref{thm. decomposition of the integral RKBS}). In Section 5, we present applications of Theorem \ref{thm. decomposition of the integral RKBS}, specifically illustrating how the size of the RKBS $\mathcal{F}_{\sigma}(\mathcal{X},\Omega)$ compares to the finite sum of p-norm RKHSs. We also discuss the existence of general solutions for one-layer neural networks (Theorem \ref{thm. general solution one-layer nn}), as well as a scheme for reformulating optimization problems in the feature space into ones in the hypothesis space (Proposition \ref{prop. reformulation of optimization problem} and Proposition \ref{prop. reformulate optimization problem of direct sum space}).

\section{Preliminaries and notations}
In this paper, we denote $I$ as a non-empty index set and the set $\{1,\dots,n\}$ is denoted by $[n]$. We consistently use $p$ and $q$ as conjugate indices, where $p$ satisfies  $1 \le p < \infty$. The data space is represented as $\mathcal{X}$, and the parameter space as $\Omega$. For convenience, we assume that $\mathcal{X}$ and $\Omega$ are compact subsets of $\mathbb{R}^{d}$ and $\mathbb{R}^{D}$ for some $d,D \in \mathbb{N}$, respectively.
We use the notation $\cong$ to denote an isomorphic relation between two vector spaces. If there exists an isometric isomorphism between two Banach spaces $B_{1}$ and $B_{2}$, we say that they are isometrically isomorphic and denote this by $B_{1}\underset{\mathcal{B}}{\cong} B_{2}$. Moreover, if $B_{1}$ and $B_{2}$ are identical as sets and the identity map itself is an isometric isomorphism, then they are completely the same. In this case, we use the notation $B_{1} \underset{\mathcal{B}}{\equiv} B_{2}$.

\subsection{Direct sum of normed vector spaces}
Let $\{a_{i}\}_{i\in I}$ be a family of elements in a Hausdorff commutative topological group (HCTG) $H$. 
Define $\mathcal{F}$ as the collection of all finite subsets of $I$, and order $\mathcal{F}$ by inclusion. Then $\mathcal{F}$ becomes a directed set.
For each $F\in\mathcal{F}$, define $a_{F}:=\sum_{i\in F}a_{i}$. Since $F$ is a finite set, $a_{F}$ would be well-defined.
Thus, $(a_{F})_{F\in\mathcal{F}}$ is a net in $H$.
The family $\{a_{i}\}_{i\in I}$ is said to be summable if the net $(a_{F})_{F\in\mathcal{F}}$ converges. In this case, the limit is called the sum of the family $\{a_{i}\}_{i\in I}$, and we denote it by $_{H}{\sum}_{i\in I}a_{i}$. When we consider sums in the norm topology of $\mathbb{R}$, we use the term $\sum_{i\in I}a_{i}$ instead of $_{\mathbb{R}}{\sum}_{i\in I}a_{i}$. The contents related to the summable family in HCTG and $\mathbb{R}$ can be found in \uppercase\expandafter{\romannumeral3} \S5 and \uppercase\expandafter{\romannumeral4} \S7 of \citet{bourbaki1971gentop} respectively.

For a given index set $I\neq \emptyset$, let $\{X_{i}\}_{i\in I}$ be a family of sets indexed by $I$. Then the direct product of the sets in $\{X_{i}\}_{i\in I}$ is defined by $\prod_{i \in I}X_{i} :=\left\{\mathbf{x}:I\rightarrow\bigcup_{i\in I}X_{i}: \mathbf{x}(i) \in X_{i} \text{ for all }i\in I \right\}$. When we assume that $X_{i} \neq \emptyset$ for all $i\in I$, by the axiom of choice, $\prod_{i\in I}X_{i}$ is the non-empty set. In this case, for $j \in I$, we can define $p_{j}: \prod_{i\in I}X_{i} \rightarrow X_{j}$ by $p_{j}(\mathbf{x}) = \mathbf{x}(j)$ for $\mathbf{x} \in \prod_{i \in I}X_{i}$. And we call $p_{j}$ is the $j$-th canonical projection. By abuse of notation, for any $\mathbf{x} \in \prod_{i\in I}X_{i}$, we denote $\mathbf{x}$ by $(x_{i})_{i\in I}$ which means $\mathbf{x}(j) = x_{j} \in X_{j}$ for all $j\in I$.
When $\{X_{i}\}_{i\in I}$ is a collection of $\mathbb{R}$-vector spaces, the direct product of $\{X_{i}\}_{i\in I}$ is the vector space $\prod_{i\in I}X_{i}$ with componentwise addition and scalar multiplication. In this case, the canonical projections are linear maps. Furthermore, if $\{X_{i}\}_{i\in I}$ are topological spaces, then we can define the direct product of $\{X_{i}\}_{i\in I}$ by giving a topology on $\prod_{i\in I}X_{i}$, called the product topology. Under this situation, the canonical projections are continuous maps.

Let $\{X_{i}:i\in I\}$ be a collection of normed vector spaces indexed by $I$. Then we can define the direct sum of the normed vector spaces $\{X_{i}:i\in I\}$ as follows:
\begin{definition}[The direct sum of normed vector spaces \citep{conway1997course}]
For $1\le p < \infty$, we define
\begin{align*}
    \bigoplus_{i\in I}^{p}X_{i}
    :=
    \left\{\mathbf{x}\in \prod_{i\in I}X_{i}
    :
    \left[\sum_{i\in I}\lVert \mathbf{x}(i)\rVert_{X_{i}}^{p}\right]^{\frac{1}{p}}<\infty \right\}
\end{align*}
as a normed vector space equipped with the norm $\lVert \mathbf{x}\rVert_{\bigoplus_{i\in I}^{p}X_{i}}
=
\left[\sum_{i\in I}\lVert \mathbf{x}(i)\rVert_{X_{i}}^{p}\right]^{\frac{1}{p}}$.
For $p = \infty$, we define
\begin{align*}
    \bigoplus_{i\in I}^{\infty}X_{i} 
    :=
    \left\{\mathbf{x}\in \prod_{i\in I}X_{i}
    :
    \sup_{i\in I}\lVert \mathbf{x}(i)\rVert_{X_{i}}<\infty \right\}
\end{align*}
as a normed vector space equipped with the norm
$\lVert \mathbf{x}\rVert_{\bigoplus_{i\in I}^{\infty}X_{i}}
=
\sup_{i\in I}\lVert \mathbf{x}(i)\rVert_{X_{i}}$.
\end{definition}
In particular, if each $X_{i}$ is a Banach space, then the direct sum of $\{X_{i}\}_{i\in I}$ is a Banach space.
Let $p$ and $q$ be conjugate indices with $1\le p < \infty$. We can obtain the following relationship between the duality and the direct sum (see \uppercase\expandafter{\romannumeral3} \S5 Exercise 4 of \citet{conway1997course}):
\begin{align}
    \label{eq. Duality of direct sum}
    \text{ Define the map }\Phi: \bigoplus_{i\in I}^{q}\left(X_{i}^{*}\right)\rightarrow\left(\bigoplus_{i\in I}^{p}X_{i}\right)^{*} \text{ as }\Phi\left((g_{i})_{i\in I}\right) (f_{i})_{i\in I}=\sum_{i\in I}\left<g_{i},f_{i}\right>
\end{align}
for $(g_{i})_{i\in I}\in \bigoplus_{i\in I}^{q}\left(X_{i}^{*}\right)$ and $(f_{i})_{i\in I}\in\bigoplus_{i\in I}^{p}X_{i}$. Then $\Phi$ is well-defined and it is an isometric isomorphism.

\subsection{Review of measure theory}
Let $K$ be a compact metric space. Then we know that the Borel and Baire $\sigma$-algebra over $K$ is coincide and every Borel measure on $K$ is Radon measure (see Proposition 6.3.4 and Theorem 7.1.7 of \citet{bogachev2007measure}).
Let $C(K)$ be the Banach space consisting of continuous real-valued functions defined on $K$, equipped with the supremum norm. We denote by $\mathcal{M}(K)$ the Banach space of (signed) Borel measures defined on $K$, endowed with the total variation norm. Additionally, the set of positive measures in $\mathcal{M}(K)$ is denoted by $\mathcal{M}(K)^{+}$, and the set of probability measures in $\mathcal{M}(K)$ is denoted by $P(K)$.
We know that by the Riesz Representation Theorem, there is an isometric isomorphism
\begin{align}
    \Lambda: \mathcal{M}(K) \rightarrow C(K)^{*} \text{ defined by }\Lambda(\mu)(f) := \int_{K}fd\mu \quad\text{for } \mu \in \mathcal{M}(K) \text{ and } f \in C(K).\label{eq. isometric isom measure space conti space}    
\end{align}

Let a measure space $(K,\Sigma,\mu)$ be given. 
We denote $\mathfrak{L}^{0}(K)$ for the set of all real valued $\Sigma$-measurable functions on $K$, that is, $\mathfrak{L}^{0}(K):=\{f:K\rightarrow \mathbb{R}: f \text{ is }\Sigma\text{-measurable}\}$. For any $1\le p \le \infty$, we define the $p$-norm by
\begin{align*}
    \lVert f\rVert_{p}:=
\begin{cases}
        \left(\int_{K}\lvert{f}\rvert d\mu\right)^{1/p}, 
       & \mbox{if } 1\le p <\infty,
       \\
       \lVert f\rVert_{\infty}= \text{ess}\sup\lvert{f}\rvert, & \mbox{if } p= \infty.
\end{cases}
\end{align*}
Then, the set $\mathfrak{L}^{p}(\mu):=\{f \in \mathfrak{L}^{0}(K): \lVert f\rVert_{p} < \infty\}$ forms a semi-normed vector space of functions. We say $f,f' \in\mathfrak{L}^{p}(\mu)$ are equivalent if and only if $\lVert f-f'\rVert_{p} =0$. Then, the set of equivalence class $L^{p}(\mu):\{[f] :f\in \mathfrak{L}^{p}(\mu)\}$ forms a Banach space with norm $\lVert [f]\rVert_{L^{p}(\mu)}:=\lVert f\rVert_{p}$.

When $p$ and $q$ are conjugate indices with $1 <p <\infty$, there is an isometric isomorphism
\begin{align}
    \Xi: L^{q}(\mu) \rightarrow L^{p}(\mu)^{*} \text{ defined by }\Xi(g)(f) := \int_{K}fgd\mu \quad\text{for }g\in L^{q}(\mu)\text{ and } f \in L^{p}(\mu).
    \label{eq. isometric isom dual of L^p and L^q}
\end{align}
It is also true for the case of $p=1, q=\infty$ if the measure space $(K,\Sigma,\mu)$ is indeed $\sigma$-finite. We use the notation $L^{p}(K,\mu)$ instead of $L^{p}(\mu)$  if there is a need to distinguish the domain space $K$. 

A family $\mathfrak{F}$ of measures in $\mathcal{M}(K)^{+}$ is called singular if $\mu \perp \nu$ whenever $\mu,\nu \in \mathfrak{F}$ and $\mu \neq \nu$ (see Definition 4.2.4 and Definition 4.6.1 of \citet{dales2016banach}).
Let $\mathfrak{S}$ be a nonempty subset of $\mathcal{M}(K)^{+}$.
Then, by Zorn's lemma, there exists a maximal element in the set $\{\mathfrak{A}: \mathfrak{A}\subset \mathfrak{S}, \text{ $\mathfrak{A}$ is a singular family in }\mathcal{M}(K)^{+}\}$. This maximal element is called a maximal singular family in $\mathfrak{S}$. Let $\{\mu_{i}\}_{i\in I}$ be a maximal singular family in $P(K)$. Then, there exists an isometric isomorphism
\begin{align}
    \Theta: \bigoplus_{i\in I}^{1}L^{1}(\mu_{i}) \rightarrow \mathcal{M}(K) \text{ defined by } \Theta\left((f_{i})_{i\in I}\right) = _{\mathcal{M}(K)}{\sum}_{i\in I}\rho_{i}
    \label{eq. M(K) is vast L^1 sum}
\end{align}
for $(f_{i})_{i\in I} \in \bigoplus_{i\in I}^{1}L^{1}(\mu_{i})$, where $\rho_{i}(B) = \int_{B}f_{i}d\mu_{i}$ for all $i\in I$ and Borel set $B$ in $K$.  (see Theorem 4.6.6 of
\citet{dales2016banach} and Proposition 4.3.8 in \citet{albiac2016topics}). We use the notation $\Phi,\Lambda,\Xi$ and $\Theta$ liberally in situations that are isometrically isomorphic, as described above.

\section{Reproducing kernel Banach spaces}
\subsection{Definition of RKBS}
When we consider $\mathbb{R}^{\mathcal{X}} = \prod_{x\in \mathcal{X}}\mathbb{R}_{x}$, where $\mathbb{R}_{x}$ is just a copy of $\mathbb{R}$ for each $x\in \mathcal{X}$, there is a natural topological structure called the product topology.
Equivalently, it is the initial topology with respect to the family of canonical projections $\{p_{x}:\mathbb{R}^{\mathcal{X}}\rightarrow \mathbb{R}_{x}\}_{x\in \mathcal{X}}$. Since this topology is compatible with the vector space structure of $\mathbb{R}^{\mathcal{X}}$, $\mathbb{R}^{\mathcal{X}}$ becomes a Hausdorff topological vector space (HTVS). Thus, we may consider a summable family $(a_{i})_{i\in I}$ in $\mathbb{R}^{\mathcal{X}}$ and denote its sum in $\mathbb{R}^{\mathcal{X}}$ by $_{\mathbb{R}^{\mathcal{X}}}{\sum}_{i \in I}a_{i}$ if it exists.

Let $V$ be a linear subspace of $\mathbb{R}^{\mathcal{X}}$. Then a topology on $V$ induced by the product topology of $\mathbb{R}^{\mathcal{X}}$ again gives $V$ the structure of a HTVS.
Additionally, due to the transitivity of the initial topology, the subspace topology on $V$ coincides with the initial topology induced by the family of restrictions $\{p_{x}|_{V}:V \rightarrow \mathbb{R}_{x}\}_{x\in \mathcal{X}}$. We denote such a HTVS as $(V,\{p_{x}|_{V}\}_{x\in \mathcal{X}})$. (Relating reference can be found in \citet{narici2010topological,bogachev2017topological,bourbaki1971gentop}). To distinguish between an index set $I$ and the data set $\mathcal{X}$, we use the term for the case of the latter as follows:
\begin{definition}
    Let $V$ be a linear subspace of $\mathbb{R}^{\mathcal{X}}$.
    For each \(x \in \mathcal{X}\), we use the term \textbf{evaluation functional at $x\in \mathcal{X}$ on $V$} to refer to the restriction of the canonical projection $p_{x}|_{V}:V \rightarrow \mathbb{R}_{x}$, denoting it as $ev_{x}$. Specifically, the function $ev_{x}: V \rightarrow \mathbb{R}$ is a linear functional defined by $ev_{x}(f) = f(x)$ for all $f \in V$.
\end{definition}
Now we define a reproducing kernel Banach space on $\mathcal{X}$ as follows:
\begin{definition}
[Definition of reproducing kernel Banach space \citep{bartolucci2023understanding,lin2022reproducing}]
    For a given set $\mathcal{X}$, a \textbf{reproducing kernel Banach space (RKBS)} $\mathcal{B}$ on $\mathcal{X}$ is a Banach space $\mathcal{B}$ of functions $f:\mathcal{X} \rightarrow \mathbb{R}$ such that
    \begin{enumerate}
        \item as a vector space, $\mathcal{B}$ is a linear subspace of $\mathbb{R}^{\mathcal{X}}$
        \item for all $x \in \mathcal{X}$, there is a constant $C_{x} \ge 0$ such that for all $f \in \mathcal{B}$, $|f(x)| \le C_{x}\lVert{f}\rVert_{\mathcal{B}}$.
    \end{enumerate}
\label{def. RKBS}
\end{definition}
According to the definition, all evaluation functionals on $\mathcal{B}$ are continuous. In other words, we have that for all $x\in \mathcal{X}$, $ev_{x} \in \mathcal{B}^{*}$.
Therefore, the norm topology of an RKBS $(\mathcal{B},\lVert\cdot\rVert_{\mathcal{B}})$  is finer than the HTVS $(\mathcal{B},\{ev_{x}\}_{x\in \mathcal{X}})$. Let $(\mathcal{B},\lVert\cdot\rVert_{1})$ and $(\mathcal{B},\lVert\cdot\rVert_{2})$ be two RKBSs on the same linear subspace $\mathcal{B}$ of $\mathbb{R}^{\mathcal{X}}$. Then by the Closed Graph Theorem, two norms $\lVert\cdot\rVert_{1}$ and $\lVert\cdot\rVert_{2}$ on the linear space $\mathcal{B}$ is equivalent (see \uppercase\expandafter{\romannumeral1} \S3 Exercise $2$ of \citet{bourbaki1953tvs} and Corollary $\text{\uppercase\expandafter{\romannumeral4}}_{1}$ of \citet{aronszajn1950theory}). In other words, when we have a function space $\mathcal{B}$, we can give an unique RKBS structure on $\mathcal{B}$ up to equivalence of norms. 

We will consider these RKBSs as hypothesis spaces in machine learning. The reason for using RKBS is as follows: When defining a hypothesis space (or function space) in machine learning, we consider completeness and pointwise convergence as the minimal assumptions required for the properties of the function space (see Chapter 1 of \citet{berlinet2011reproducing}).

\subsection{Characterization of RKBSs}
Before we state the characterization theorem of RKBSs, we introduce a method that induces a mathematical structure from a pre-existing structure. Let $V$ be a normed vector space over $\mathbb{R}$ equipped with the norm $\lVert\cdot\rVert_{V}$, and let $W$ be a vector space over $\mathbb{R}$. If there is a vector space isomorphism $T:V \rightarrow W$, then $\lVert T^{-1}(\cdot)\rVert_{V}:W\rightarrow \mathbb{R}$ defines a norm on $W$. Furthermore, when we consider $W$ as a normed vector space equipped with the norm $\lVert T^{-1}(\cdot)\rVert_{V}$, the linear isomorphism $T:(V,\lVert\cdot\rVert_{V}) \rightarrow (W,\lVert T^{-1}(\cdot)\rVert_{V})$ becomes an isometric isomorphism (This is referred to as the transport of structure).

Let $V$ and $W$ be vector spaces. If $T: V\rightarrow W$ is a linear map, then there exists an unique linear map $\hat{T}: V/\ker{T} \rightarrow W$ such that $\hat{T}\circ\pi = T$, where $\pi: V \rightarrow V/\ker{T}$ defined by $\pi(v) = [v]$ for $v \in V$. Throughout this paper, we use the notation $\hat{T}$ to denote the induced linear map described above in similar situations. 
We now state the characterization theorem of RKBSs introduced by \citet{combettes2018regularized}.
\begin{theorem}[Characterization of RKBSs \citep{bartolucci2023understanding,combettes2018regularized}]
    A linear subspace $\mathcal{B}$ of $\mathbb{R}^{\mathcal{X}}$ is an RKBS on $\mathcal{X}$ if and only if there exists a Banach space $\Psi$ and a map $\psi:\mathcal{X} \rightarrow \Psi^{*}$ such that $\mathcal{B} = \operatorname{im}(A) = \{f: \exists \nu \in \Psi \text{ s.t. } A(\nu) = f\}$ with the norm $\lVert{f}\rVert_{\mathcal{B}} = \underset{\nu \in A^{-1}(f)}{\inf}\lVert{\nu}\rVert_{\Psi}$,
    where $A:\Psi \rightarrow \mathbb{R}^{\mathcal{X}}$ is a linear map defined by $(A\nu)(x) := \left<\psi(x),\nu\right>$ for $x \in \mathcal{X}$ and $\nu \in \Psi$.
    \label{thm. RKBS Characterization}
\end{theorem}
Note that the linear map $A$ is the linear transformation induced from the family of the linear maps $\{\psi(x):\Psi \rightarrow \mathbb{R}_{x}\}_{x\in X}$ by the universal property of the direct product of the vector spaces $\{\mathbb{R}_{x}\}_{x\in \mathcal{X}}$.
We briefly review the proof provided in \citet{bartolucci2023understanding}.
In the necessity part of the proof, it is shown that $\ker{A}$ is closed in $\Psi$ by the following equations:
\begin{align}
    \ker(A) &=\{\nu\in\Psi: \psi(x)(\nu) = 0 \text{ for all }x \in \mathcal{X}\}
    =\bigcap_{x\in{X}}\ker{\psi(x)}.
    \label{eq. Characterziation thm kerA closed}
\end{align}
Thus, $\Psi/\ker{A}$ can be a Banach space with the quotient norm.
Consider the linear map $\hat{A}:\Psi/\ker{A} \rightarrow \mathbb{R}^{\mathcal{X}}$ such that $A = \hat{A}\circ\pi$. Since $\hat{A}:\Psi/\ker{A}\cong \operatorname{im}(A)$ is an isomorphism of vector spaces, by the transport of the structure, $\mathcal{B}=\operatorname{im}(A)$ becomes a Banach space with the norm:
\begin{align*}
    \lVert{f}\rVert_{\mathcal{B}}
    =
    \lVert\hat{A}^{-1}(f)\rVert_{\Psi/
    \ker(A)}
    =
    \inf_{\nu \in \pi^{-1}(\hat{A}^{-1}(f))}\lVert \nu \rVert_{\Psi} 
    =
    \inf_{\nu \in A^{-1}(f)}\lVert \nu \rVert_{\Psi}     
\end{align*}
The evaluation functionals are continuous as follows:  
for any $f \in \mathcal{B}$ and $\nu \in A^{-1}(f)$, we have $|f(x)| = |A\nu(x)| 
\le
\lVert{\psi(x)}\rVert_{\Psi^{*}}\lVert\nu\rVert_{\Psi}$.
Thus, we can deduce that for all $x\in \mathcal{X}$,
\begin{align}
    \lVert ev_{x}(f)\rVert_{\mathbb{R}}=|f(x)| \le \lVert{\psi(x)}\rVert_{\Psi^{*}}\inf_{\nu \in A^{-1}(f)}\lVert\nu\rVert_{\Psi}
    =\lVert{\psi(x)}\rVert_{\Psi^{*}}\lVert f\rVert_{\mathcal{B}}.
    \label{eq. evaluation functional is bounded}
\end{align}

From now on, for a given RKBS $\mathcal{B}$, we consider a corresponding space $\Psi$, a map $\psi$ and an induced linear map $A$. In this situation, by abuse of notation, we may say that an RKBS triple $\mathcal{B}=(\Psi,\psi,A)$ is given. Each component of the triple $(\Psi,\psi,A)$ has a specific name. Specifically, we refer to $\Psi$ as a feature space, $\psi$ as a feature map, and $A$ as an RKBS map in order.
\subsection{One-layer neural networks}
In this subsection, we assume that $\Omega_{1} \subset \mathbb{R}^{d}$ and $\Omega_{2} \subset \mathbb{R}$ are compact, and let $\Omega = \Omega_{1} \times \Omega_{2}$. Consider a continuous nonlinear function $g:\mathbb{R} \rightarrow \mathbb{R}$. The prediction function represented by a one-layer neural network with a one-dimensional target can be expressed as follows:
\begin{align}
    \label{eq. one-layer nn}
    f(x) = \sum_{i=1}^{m}\eta_{i}g(x\cdot\theta_{i}-b_{i}),
\end{align}
where $x\in \mathcal{X}$, $\theta_{i} \in \Omega_{1}$, $b_{i} \in \Omega_{2}$ and $\eta_{i} \in \mathbb{R}$ for $i=1,\dots,m$. 
For convention, and with some abuse of notation, we define a continuous function $\sigma:\mathcal{X}\times\Omega \rightarrow \mathbb{R}
$ by $\sigma(x,w)=g(x\cdot\theta-b)$ where $w = (\theta,b)$. This gives the following simplified representation: $f(x)=\sum_{i=1}^{m}\eta_{i}\sigma(x,w_{i})$. Using measure-theoretic notation, we can have an integral representation of \eqref{eq. one-layer nn}: $f(x) = \int_{\Omega}\sigma(x,w)d\left(\sum_{i=1}^{m}\eta_{i}\delta_{w_{i}}\right)$ where $\delta_{w_{i}}$ is the Dirac measure at $w_{i}$. When considering the limit as $m \rightarrow \infty$ in \eqref{eq. one-layer nn}, we obtain the following: 
\begin{align*}
    \int_{\Omega}\sigma(x,w)d\left(\sum_{i=1}^{m}\eta_{i}\delta_{w_{i}}\right) \rightarrow \int_{\Omega}\sigma(x,w)d\mu(w),
\end{align*}
for some $\mu \in \mathcal{M}(\Omega)$. A more detailed explanation can be found in Chapter 9 of \citet{bach2024learning}. In the following subsection, we will define the hypothesis space of one-layer neural networks in a more abstract way using this relaxed expression.

\subsection{Integral RKBS and p-norm RKBS}

Directly using the characterization theorem, we can define the hypothesis spaces that are considered to represent one-layer neural networks. Until section 4, we consider a fixed element $\sigma$ in $C(\mathcal{X}\times\Omega)$, where $\mathcal{X}$ is a compact subset of $\mathbb{R}^{d}$ and $\Omega$ is a compact subset of $\mathbb{R}^{D}$ for some $d,D\in \mathbb{N}$.

Let $V$ and $W$ be a real normed vector spaces. If we denote $V^{**}$ be the bidual space of $V$, then there is a linear isometric embedding $\iota: V\rightarrow V^{**}$, called the canonical embedding of $V$ in $V^{**}$, defined by $\iota(v)(v^{*})= v^{*}(v)$ for $v\in V$ and $v^{*} \in V^{*}$.
For a given bounded linear operator $T:V \rightarrow W$, the dual operator of $T$ is the linear operator $T^{*}:W^{*} \rightarrow V^{*}$ defined by $T^{*}(w^{*}):= w^{*} \circ T$ for $w^{*} \in W^{*}$.

\begin{definition}[A class of integral RKBSs, associated with the function $\sigma$ \citep{bartolucci2023understanding,spek2022duality}]
    Let $\mathcal{M}(\Omega)$ be a feature space. Consider a feature map $\psi:\mathcal{X} \rightarrow \mathcal{M}(\Omega)^{*}(\underset{\mathcal{B}}{\cong} C(\Omega)^{**})$ defined by $\psi(x) = \Lambda^{*}(\iota(\sigma(x,\cdot)))$ for all $x \in \mathcal{X}$, where $\iota:C(\Omega)\rightarrow C(\Omega)^{**}$ is the canonical embedding of $C(\Omega)$ in $C(\Omega)^{**}$ and $\Lambda^{*}$ is the dual operator of $\Lambda: \mathcal{M}(\Omega) \rightarrow C(\Omega)^{*}$, which is defined in \eqref{eq. isometric isom measure space conti space}. 
    Then there is a linear map $A:\mathcal{M}(\Omega)\rightarrow \mathbb{R}^{\mathcal{X}}$ defined by $(A\mu)(x) = \left<\psi(x),\mu\right> = \int_{\Omega}\sigma(x,w)d\mu(w)$ for $x\in \mathcal{X}$ and $\mu \in \mathcal{M}(\Omega)$. An integral RKBS $\mathcal{F}_{\sigma}(\mathcal{X},\Omega)$, associated with the function $\sigma$ is defined by the Banach space
    \begin{align}
        \mathcal{F}_{\sigma}(\mathcal{X},\Omega)
        &:=\left\{f\in \mathbb{R}^{\mathcal{X}}:\exists\mu\in\mathcal{M}(\Omega) \text{ s.t. } \forall x\in \mathcal{X}, f(x)=\int_{\Omega}\sigma(x,w)d\mu(w)\right\},
    \end{align}
    equipped with the norm $\lVert{f}\rVert_{\mathcal{F}_{\sigma}(X,\Omega)} = \inf_{\mu\in A^{-1}(f)}\lVert{\mu}\rVert_{\mathcal{M}(\Omega)}$.
    \label{def. integral RKBS}
\end{definition}

In the above Definition \ref{def. integral RKBS}, consider the linear map $A:\mathcal{M}(\Omega)\rightarrow \mathbb{R}^{\mathcal{X}}$.
We deduce that, by the Dominated Convergence Theorem, $\operatorname{im}(A)$ is a linear subspace of $C(\mathcal{X})$ (see Theorem 2.27 of \citet{folland1999real}). Furthermore, from the inequality $\left\lVert A\mu\right\rVert_{C(\mathcal{X})} 
\le \sup_{x\in \mathcal{X},w \in \Omega}\left\lvert\sigma(x,w)\right\rvert \lVert\mu\rVert$, we can see that the map $A:\mathcal{M}(\Omega) \rightarrow C(\mathcal{X})$ is indeed a bounded operator. Recently, Steinwart showed that when $\mathcal{X}$ is an uncountable compact metric space, there is no RKHS $\mathcal{H}$ on $\mathcal{X}$ such that $C(\mathcal{X}) \subset \mathcal{H}$ (see \citet{steinwart2024reproducing}). We can obtain a similar result for the class of integral RKBS as well.
\begin{proposition}\label{prop. A is compact operator}
    The bounded operator $A:\mathcal{M}(\Omega) \rightarrow C(\mathcal{X})$
    defined by
    \begin{align*}
        (A\mu)(x) = \int_{\Omega}\sigma(x,w)d\mu(w)    
    \end{align*} 
    for $x\in \mathcal{X}$ and $\mu \in \mathcal{M}(\Omega)$ is compact.
\end{proposition}
\begin{proof}
    Since $\Omega$ is a compact metric space, $C(\Omega)$ is separable space. Let $\{\mu_{n}\}$ be a bounded sequence in $\mathcal{M}(\Omega) \underset{\mathcal{B}}{\cong} C(\Omega)^{*}$. Then, by the separable version of the Banach-Alaoglu Theorem, there exists a weak* convergent subsequence $\{\mu_{n_{k}}\}$ such that $\mu_{n_{k}} \underset{w^{*}}{\longrightarrow} \mu$ (see Problem 10 of Chapter 4.9 in \citet{kreyszig1991introductory}).
    Define $\Gamma :=\{\sigma(x,\cdot)\in C(\Omega): x\in \mathcal{X}\}$.
    Since $\Gamma$ is uniformly bounded and pointwise equicontinuous, we have the following (see Exercise 8.10.134 of \citet{bogachev2007measure}):
    \begin{align*}
        \lim_{n\rightarrow\infty}\left\lVert A\mu_{n_{k}}-A\mu\right\rVert_{C(\mathcal{X})}
        &=
        \lim_{n\rightarrow\infty}\sup_{x\in \mathcal{X}}\left\lvert\int \sigma(x,w)d(\mu_{n_{k}}-\mu)(w)\right\rvert
        \\
        &=
        \lim_{n\rightarrow\infty}\sup_{f\in\Gamma}\left\lvert\int fd(\mu_{n_{k}}-\mu)\right\rvert
        = 0.
    \end{align*}
\end{proof}

Using the proposition above, it follows that if $\operatorname{im}(A)$ is closed in $C(\mathcal{X})$, then $\operatorname{im}(A)$ has finite dimension. Thus, when $\mathcal{X}$ is an infinite compact metric space, we deduce that $\mathcal{F}_{\sigma}(\mathcal{X},\Omega) \subsetneq C(\mathcal{X})$, and in general, $\mathcal{F}_{\sigma}(\mathcal{X},\Omega)$ cannot be a Banach space if it equipped with the supremum norm.

\begin{definition}[A class of p-Norm RKBS, associated with the function $\sigma$ \citep{spek2022duality}]
    Let $\pi \in P(\Omega)$ be given. Let $p$ and $q$ be conjugate indices such that $1\le p < \infty$.
    Take a feature space $\Psi$ as $L^{p}(\pi)$ and choose a feature map $\psi: \mathcal{X} \rightarrow (L^{p}(\pi))^{*}$ defined by $\psi(x) = \Xi(\sigma(x,\cdot))$ for $x\in \mathcal{X}$, where $\Xi: L^{q}(\pi)\rightarrow (L^{p}(\pi))^{*}$ is the isometric isomorphism defined in \eqref{eq. isometric isom dual of L^p and L^q}\footnote{ We wrote
    $\psi(x) = \Xi(\sigma(x,\cdot))$ as a slight abuse of notation. More precisely, $\psi(x) = \Xi(T(\sigma(x,\cdot)))$ where $T:\mathfrak{L}^{p}(\pi) \rightarrow \mathfrak{L}^{p}(\pi)/N\underset{\mathcal{B}}{\cong} L^{p}(\pi)$ is the canonical quotient map, and $N:=\{ f \in\mathfrak{L}^{p}(\pi):\lVert f\rVert_{p}=0\}$ is the subspace of functions vanishing almost everywhere.}. Then, there is a linear map $A:L^{p}(\pi)\rightarrow \mathbb{R}^{\mathcal{X}}$ defined by $(Ah)(x) = \left<\psi(x),h\right>$ for $x\in \mathcal{X}$ and $h \in L^{p}(\pi)$.
    We define a p-Norm RKBS $\mathcal{L}_{\sigma}^{p}(\pi)$, associated with the function $\sigma$ by the Banach space
    \begin{align}
        \mathcal{L}_{\sigma}^{p}(\pi)
        :=
        \left\{f \in \mathbb{R}^{\mathcal{X}}: \exists h \in L^{p}(\pi) \text{ s.t. }\forall x\in \mathcal{X}, f(x)=\int_{\Omega}\sigma(x,w)h(w)d\pi(w) \right\},
    \end{align}
    equipped with the norm $\lVert{f}\rVert_{\mathcal{L}_{\sigma}^{p}(\pi)} = \inf_{h \in A^{-1}(f)}\lVert{h}\rVert_{L^{p}(\pi)}$.
    \label{def. p-norm RKBS}
\end{definition}

When we consider the $p=2$ case, we obtain the RKHS $\mathcal{L}_{\sigma}^{2}(\pi)$. This space corresponds to $\mathcal{F}_{2}$ as described in \citet{bach2017breaking}. The kernel of $\mathcal{L}_{\sigma}^{2}(\pi)$ is given by $k(x,y) = \int_{\Omega}\sigma(x,w)\sigma(y,w)d\pi(w)$ for $(x,y) \in \mathcal{X}\times\mathcal{X}$.
Furthermore, $\mathcal{L}_{\sigma}^{2}(\pi)$ is embedded in $\mathcal{L}_{\sigma}^{1}(\pi)$ (that is, $\mathcal{L}_{\sigma}^{2}(\pi)\subset \mathcal{L}_{\sigma}^{1}(\pi)$) and for all $f \in \mathcal{L}_{\sigma}^{2}(\pi)$,  $\lVert f\rVert_{\mathcal{L}_{\sigma}^{1}(\pi)} \le \lVert f\rVert_{\mathcal{L}_{\sigma}^{2}(\pi)}$.
As an analogue to the case of $L^{p}$ space, we sometimes use the notation $\mathcal{L}_{\sigma}^{p}(\Omega,\pi)$ instead of $\mathcal{L}_{\sigma}^{p}(\pi)$ to avoid confusion.

\subsection{Infinite sum of reproducing kernel Banach spaces}
Let an RKBS $\mathcal{B}$ be given. If we consider an evaluation functional on $\mathcal{B}$ evaluating at $x\in \mathcal{X}$ by $ev_{x}:\mathcal{B} \rightarrow \mathbb{R}$, as discussed earlier, then we have that for all $x\in \mathcal{X}$, $ev_{x} \in \mathcal{B}^{*}$.
Thus, if we assume that a collection of RKBSs $\{\mathcal{B}_{i}\}_{i\in I}$ is given and denote $ev_{x}^{i}$ as the evaluation functional on $\mathcal{B}_{i}$ evaluating at $x\in \mathcal{X}$, then for all $i \in I$ and $x \in \mathcal{X}$, $ev_{x}^{i} \in \mathcal{B}_{i}^{*}$. Now, we define the sum of RKBSs as follows, modifying Example 3.13 in \citet{combettes2018regularized} and the theorem on page 353 of \citet{aronszajn1950theory}:
\begin{proposition}[Infinite sum of reproducing kernel Banach spaces]
    Let $p$ and $q$ be conjugate indices with $1\le p <\infty$. Let $\{B_{i}\}_{i\in I}$ be a collection of Reproducing Kernel Banach Spaces on $\mathcal{X}$. If 
    $(ev_{x}^{i})_{i\in I} \in \bigoplus_{i\in I}^{q}\mathcal{B}_{i}^{*}$ for all $x \in \mathcal{X}$, then
    $\mathcal{B} = \left\{ _{\mathbb{R}^{\mathcal{X}}}{\sum}_{i \in I}f_{i}: (f_{i})_{i\in I} \in \bigoplus_{i\in I}^{p}\mathcal{B}_{i} \right\}$ is an RKBS on $\mathcal{X}$ with the norm $\lVert f\rVert_{\mathcal{B}} 
    =
    \inf_{f = _{\mathbb{R}^{\mathcal{X}}}{\sum}_{i \in I}f_{i}}\lVert(f_{i})_{i\in I}\rVert_{\bigoplus_{i\in I}^{p}\mathcal{B}_{i}}$.
    \label{prop. infinite sum of RKBSs}
\end{proposition}

\begin{proof} 
    Let $\bigoplus_{i\in I}^{p}\mathcal{B}_{i}$ be a feature space and define a feature map $\mathbf{s}:\mathcal{X} \rightarrow \left(\bigoplus_{i\in I}^{p}\mathcal{B}_{i}\right)^{*}$ as $\mathbf{s}(x) = \Phi((ev_{x}^{i})_{i\in I})$ for $x\in \mathcal{X}$, where $\Phi:\bigoplus_{i\in I}^{q}\mathcal{B}_{i}^{*}\rightarrow\left( \bigoplus_{i\in I}^{p}\mathcal{B}_{i}\right)^{*}$ is the isometric isomorphism defined in \eqref{eq. Duality of direct sum}.
    Now, there is a linear transformation
    $\mathcal{S}: \bigoplus_{i\in I}^{p}\mathcal{B}_{i} \rightarrow \mathbb{R}^{\mathcal{X}}$ by $\left(\mathcal{S}(f_{i})_{i\in I}\right)(x) = \left<\mathbf{s}(x),(f_{i})_{i \in I}\right>$ for $(f_{i})_{i \in I} \in \bigoplus_{i\in I}^{p}\mathcal{B}_{i}$ and $x \in \mathcal{X}$. Then, by Theorem \ref{thm. RKBS Characterization}, $\bigoplus_{i\in I}^{p}\mathcal{B}_{i}/ \ker(\mathcal{S}) = \operatorname{im}(\mathcal{S})$ is an RKBS on $\mathcal{X}$ with the norm $\lVert f\rVert_{\mathcal{B}} = \inf_{(f_{i})_{i\in I}\in \mathcal{S}^{-1}(f)}\lVert(f_{i})_{i\in I}\rVert_{\bigoplus_{i\in I}^{p}\mathcal{B}_{i}}$.
\end{proof}

Note that the property of net in the initial topology implies that $f = _{\mathbb{R}^{\mathcal{X}}}{\sum}_{i \in I}f_{i}$ in $(\mathbb{R}^{\mathcal{X}},\{p_{x}\}_{x\in \mathcal{X}})$ is equivalent to $f(x) = \sum_{i \in I}f_{i}(x)$ for all $x \in \mathcal{X}$. Thus, we have that
\begin{align*}
    \mathcal{B} 
    &=
    \left\{f \in \mathbb{R}^{\mathcal{X}}: \exists (f_{i})_{i\in I} \in \bigoplus_{i\in I}^{p}\mathcal{B}_{i} \text{ s.t. }\forall x\in \mathcal{X}, f(x)= \sum_{i\in I}\left< ev_{x}^{i},f_{i}\right> \right\}
    \\
    &=
    \left\{ _{\mathbb{R}^{\mathcal{X}}}{\sum}_{i \in I}f_{i}: (f_{i})_{i\in I} \in \bigoplus_{i\in I}^{p}\mathcal{B}_{i} \right\}.
\end{align*}
From now on, we denote $\mathcal{B}$ mentioned in Proposition \ref{prop. infinite sum of RKBSs} by $\sum_{i\in I}^{p}\mathcal{B}_{i}$ and call it the sum of the family of RKBSs $\{\mathcal{B}_{i}\}_{i\in I}$. In particular, for the case of $p =1$, we denote $\mathcal{B}$ as $\sum_{i\in I}\mathcal{B}_{i}$. In Proposition \ref{prop. infinite sum of RKBSs}, we intentionally use the notation $\mathbf{s}$ for the feature map and $\mathcal{S}$ for the RKBS map to emphasize that they are used to represent the sum of RKBSs. With a necessary abuse of notation, we define the sum of RKBSs via an RKBS triple as $\sum_{i\in I}^{p}\mathcal{B}_{i} = (\bigoplus_{i\in I}^{p}\mathcal{B}_{i}, \mathbf{s}, \mathcal{S})$ where $\mathbf{s}$ and $\mathcal{S}$ are defined in the proof of Proposition \ref{prop. infinite sum of RKBSs}.
\begin{remark}
    Let a family of RKBS triples $ 
    \{\mathcal{B}_{i}=(\Psi_{i},\psi_{i},A_{i})\}_{i\in I}$ be given. From \eqref{eq. evaluation functional is bounded}, we know that $\lVert ev_{x}^{i} \rVert_{B_{i}^{*}} \le \lVert \psi_{i}(x)\rVert_{\Psi_{i}^{*}}$ for all $x \in \mathcal{X}$ and $i\in I$.
    Thus, instead of assuming $(ev_{x}^{i})_{i\in I} \in \bigoplus_{i\in I}^{q}\mathcal{B}_{i}^{*}$  for all $x \in \mathcal{X}$, it suffices to assume that $(\psi_{i}(x))_{i\in I} \in \bigoplus_{i\in I}^{q}\Psi_{i}^{*}$ for all $x\in \mathcal{X}$.\label{remark. inf sum condition relate feature map}
\end{remark}

\section{Main Results} 

\subsection{Compatibility between the sum of RKBSs and the direct sum of feature spaces}

In this section, we present the compatibility between the sum of RKBSs and the direct sum of feature spaces. Before stating our main proposition, we prove the following lemma, which says that the restriction to the direct sum of Banach spaces of the product of (quotient, isometry) maps preserves their properties.
 
\begin{lemma}\label{lem. for compatibility} 
Suppose $1\le p <\infty$, and let a family of Banach spaces $\{X_{i}\}_{i \in I}$ be given.
\begin{enumerate}
    \item Suppose that for each $i\in I$, $D_{i}$ is a closed linear subspace of $X_{i}$, and let $\pi_{i}:X_{i} \rightarrow X_{i}/D_{i}$ be the projection map defined by $\pi_{i}(x_{i}):= \left[ x_{i}\right]$ for $x_{i} \in X_{i}$. Then, the map $\widetilde{\left( \pi_{i}\right)_{i\in I}}:\bigoplus_{i\in I}^{p}X_{i} \rightarrow \bigoplus_{i\in I}^{p}X_{i}/D_{i}$ defined by $\widetilde{\left(\pi_{i}\right)_{i\in I}}\left((x_{i})_{i\in I}\right) = \left(\pi_{i}(x_{i})\right)_{i\in I}$ for $(x_{i})_{i \in I}\in \bigoplus_{i\in I}^{p}X_{i}$ is a surjective bounded linear operator.

    \item Assume there is another family of Banach spaces $\{ Y_{i}\}_{i \in I}$. If for each $i\in I$, there is an isometric isomorphism $\phi_{i}: X_{i} \rightarrow Y_{i}$, then the map $\widetilde{\left( \phi_{i}\right)_{i \in I}}: \bigoplus_{i\in I}^{p}X_{i} \rightarrow \bigoplus_{i\in I}^{p}Y_{i}$ defined by $\widetilde{\left(\phi_{i}\right)_{i\in I}}\left((x_{i})_{i\in I}\right) = \left(\phi_{i}(x_{i})\right)_{i\in I}$ for $(x_{i})_{i \in I}\in \bigoplus_{i\in I}^{p}X_{i}$ is an isometric isomorphism.
\end{enumerate}
\end{lemma}
\begin{proof}
    We know that for each $i\in I$, $\pi_{i}$ is a surjective bounded linear operator, and its norm satisfies $\lVert \pi_{i}\rVert\le 1$ (see \uppercase\expandafter{\romannumeral3} \S4 Theorem 4.2 of \citet{conway1997course}). Additionally, there is an unique linear map $\left(\pi_{i}\right)_{i \in I}: \prod_{i\in I}X_{i} \rightarrow \prod_{ i \in I}X_{i}/D_{i}$ such that $\pi_{j}\circ p_{j} = q_{j} \circ \left(\pi_{i}\right)_{i \in I}$ for all $j \in I$, where $p_{j}$ and $q_{j}$ are $j$-th canonical projections of $\prod_{i\in I}X_{i}$ and $\prod_{i\in I}X_{i}/D_{i}$, respectively. 
    Consider the restriction of $\left(\pi_{i}\right)_{i \in I}$ to $\bigoplus_{i\in I}^{p}X_{i}$ and denote it by $\widetilde{\left(\pi_{i}\right)_{i\in I}}: \bigoplus_{i\in I}^{p}X_{i} \rightarrow \prod_{i\in I}X_{i}/D_{i}$.
    Let $\left(x_{i}\right)_{i\in I} \in \bigoplus_{i\in I}^{p}X_{i}$.
    Since $\widetilde{\left(\pi_{i}\right)_{i\in I}}\left((x_{i})_{i\in I}\right) = \left(\pi_{i}(x_{i})\right)_{i\in I} \in \prod_{i\in I}X_{i}/D_{i}$ and $\sum_{i\in I}\lVert \pi_{i}(x_{i})\rVert_{X_{i}/D_{i}}^{p} \le \sum_{i \in I}\lVert x_{i}\rVert_{X_{i}}^{p} < \infty$, it follows that $\operatorname{im}\left(\widetilde{\left(\pi_{i}\right)_{i\in I}}\right) \subset \bigoplus_{i\in I}^{p}X_{i}/D_{i}$. From this, we also know that $\widetilde{\left(\pi_{i}\right)_{i\in I}}$ is a bounded operator with norm less than $1$.
    
    It remains to show the surjectivity of $\widetilde{\left( \pi_{i}\right)_{i\in I}}:\bigoplus_{i\in I}^{p}X_{i} \rightarrow \bigoplus_{i\in I}^{p}X_{i}/D_{i}$.
    Let $(\pi_{i}(x_{i}))_{i \in I} \in \bigoplus_{i\in I}^{p}X_{i}/D_{i}$. Then, we have $\sum_{i \in I}\inf_{d_{i} \in D_{i}}\lVert x_{i} + d_{i}\rVert_{X_{i}}^{p} 
    =
    \sum_{i \in I}\left(\inf_{d_{i} \in D_{i}}\lVert x_{i} + d_{i}\rVert_{X_{i}}\right)^{p} 
    =
    \sum_{i \in I}\lVert \pi_{i}(x_{i})\rVert_{X_{i}/D_{i}}^{p}
    < \infty$ and the set $N =\{i \in I: \lVert \pi_{i}(x_{i})\rVert_{X_{i}/D_{i}} > 0\}$ is countable. Let $f:\mathbb{N} \rightarrow N$ be a reordering bijection. From the definition of the infimum, for each $k\in \mathbb{N}$, we can take $\Tilde{d}_{f(k)} \in D_{f(k)}$ such that 
    \begin{align*}
        \lVert x_{f(k)} + \Tilde{d}_{f(k)}\rVert_{X_{f(k)}}^{p}
        <
        \inf_{d_{f(k)} \in D_{f(k)}}\lVert x_{f(k)} + d_{f(k)}\rVert_{X_{f(k)}}^{p} + \frac{1}{k^{2}}.
    \end{align*}
    Then, we have that:
    \begin{align*}
        \sum_{i \in N}\lVert x_{i} + \Tilde{d}_{i}\rVert_{X_{i}}^{p} 
        = 
        \sum_{k=1}^{\infty}\lVert x_{f(k)} + \Tilde{d}_{f(k)}\rVert_{X_{f(k)}}^{p}
        <
        \sum_{k=1}^{\infty}\inf_{d_{f(k)} \in D_{f(k)}}\lVert x_{f(k)} + d_{f(k)}\rVert_{X_{f(k)}}^{p} + \sum_{k=1}^{\infty}\frac{1}{k^{2}}
        <
        \infty.
    \end{align*}
    Thus, if we take $x_{i}'=
    \begin{cases}
    x_{i}+\Tilde{d}_{i} & \mbox{if }i\in N, \\
    0 & \mbox{if } i \in I\setminus N
    \end{cases}$, then $\left(x_{i}'\right)_{i\in I} \in \bigoplus_{i \in I}^{p}X_{i}$ and $\widetilde{\left(\pi_{i}\right)_{i\in I}}\left( (x_{i}')_{i\in I}\right) = \left( \pi_{i}(x_{i})\right)_{i \in I}$.
    We can also prove (2) directly.

\end{proof}

The following proposition is one of the main result of this paper. It states that when we have a family of RKBSs, there is an RKBS induced by the direct sum of feature spaces, which is isometrically isomorphic to the sum of the given family of RKBSs. Conversely, when we have an RKBS induced by the direct sum of feature spaces, there is a collection of RKBSs such that their sum is isometrically isomorphic to the given RKBS.

\begin{proposition}[Compatibility]
\label{prop. compatibility}
Let $I \neq \emptyset$ be an index set. Let $p$ and $q$ be conjugate indices, where $p$ satisfies $1\le p < \infty$.
\begin{enumerate}
    \item Suppose a family of RKBS triples $ 
    \{\mathcal{B}_{i}=(\Psi_{i},\psi_{i},A_{i})\}_{i\in I}$ is given and $(ev_{x}^{i})_{i\in I}\in \bigoplus_{i\in I}^{q}\mathcal{B}_{i}^{*}$ for all $x \in \mathcal{X}$.
    Then, there is an RKBS triple $\mathcal{B}= (\bigoplus_{i\in I}^{p}\Psi_{i},\psi,A)$ such that $\mathcal{B} \underset{\mathcal{B}}{\equiv} \sum_{i\in I}^{p}\mathcal{B}_{i}$.
    \item For an RKBS triple $\mathcal{B}= (\bigoplus_{i\in I}^{p}\Psi_{i},\psi,A)$, there is a family of RKBS triples $\{\mathcal{B}_{i} = (\Psi_{i},\psi_{i},A_{i})\}_{i\in I}$ such that $\mathcal{B} \underset{\mathcal{B}}{\equiv} \sum_{i\in I}^{p}\mathcal{B}_{i}$.
\end{enumerate}
\end{proposition}
\begin{figure}[ht]
\centering
\begin{tikzcd}[sep=huge]
    \bigoplus_{i\in I}^{p}\Psi_{i} \arrow[r, "\widetilde{(\pi_{i})_{i\in I}}"] \arrow[dr, "A"']
    &  \bigoplus_{i\in I}^{p}\Psi_{i}/\ker{A_{i}}
    \arrow[r, "\widetilde{(\hat{A}_{i})_{i\in I}}"]
    &
    \bigoplus_{i\in I}^{p}\mathcal{B}_{i}
    \arrow[dl, "\mathcal{S}"]
    \\
    &
    \mathbb{R}^{\mathcal{X}}
\end{tikzcd}
\caption{Commutative diagram for the compatibility}
\label{fig: Commutative diagram for the compatibility}
\end{figure}
The diagram above intuitively illustrates the result we aim to demonstrate in Proposition \ref{prop. compatibility}. Detailed information about each map can be found in the proof.

\begin{proof}
    From Lemma \ref{lem. for compatibility}, we know that there is a surjective bounded linear operator $\widetilde{(\pi_{i})_{i\in I}}: \bigoplus_{i \in I}^{p}\Psi_{i} \rightarrow \bigoplus_{i\in I}^{p}\Psi_{i}/\ker{A_{i}}$ and an isometric isomorphism $\widetilde{(\hat{A}_{i})_{i\in I}}:\bigoplus_{i\in I}^{p}\Psi_{i}/\ker{A_{i}} \rightarrow \bigoplus_{i\in I}^{p}\mathcal{B}_{i}$. Let $\Phi:\bigoplus_{i\in I}^{q}\mathcal{B}_{i}^{*}\rightarrow \left(\bigoplus_{i\in I}^{p}\mathcal{B}_{i}\right)^{*}$ be the isometric isomorphism defined in \eqref{eq. Duality of direct sum}.
    Since $(ev_{x}^{i})_{i\in I}\in \bigoplus_{i\in I}^{q}\mathcal{B}_{i}^{*}$ for all $x \in \mathcal{X}$, we can apply Proposition \ref{prop. infinite sum of RKBSs} to deduce that there is an RKBS triple for the summation of RKBSs $\sum_{i\in I}^{p}\mathcal{B}_{i} = (\bigoplus_{i\in I}^{p}\mathcal{B}_{i}, \mathbf{s}, \mathcal{S})$. Consider the map $A:= \mathcal{S}\circ \widetilde{(\hat{A}_{i})_{i\in I}}\circ\widetilde{(\pi_{i})_{i\in I}}
    =
    \mathcal{S}\circ \widetilde{(A_{i})_{i\in I}}: \bigoplus_{i \in I}^{p}\Psi_{i} \rightarrow \mathbb{R}^{\mathcal{X}}$. To verify the map $A$ is indeed an RKBS map, we show the following holds
    \begin{align*}
        \left(A(\mu_{i})_{i\in I}\right)(x)
        =
        \left( \mathcal{S}\left(\widetilde{(A_{i})_{i\in I}} (\mu_{i})_{i\in I}\right)\right)(x)
        =
        \left<\Phi((ev_{x}^{i})_{i\in I}) \circ \widetilde{(A_{i})_{i\in I}}, (\mu_{i})_{i\in I}\right>
    \end{align*}
    for all $x \in \mathcal{X}$ and $(\mu_{i})_{i \in I} \in \bigoplus_{i \in I}^{p}\Psi_{i}$.
    Thus, if we define a feature map $\psi: \mathcal{X} \rightarrow \left(\bigoplus_{i \in I}^{p}\Psi_{i}\right)^{*}$ by $\psi(x)= \Phi((ev_{x}^{i})_{i\in I}) \circ \widetilde{(A_{i})_{i\in I}} \in \left(\bigoplus_{i \in I}^{p}\Psi_{i}\right)^{*}$ for $x \in \mathcal{X}$, then we get an RKBS triple $\mathcal{B}= (\bigoplus_{i\in I}^{p}\Psi_{i},\psi,A)$.
    Since $\widetilde{(\hat{A}_{i})_{i\in I}}\circ\widetilde{(\pi_{i})_{i\in I}}$ is surjective, $\operatorname{im}(A) = \operatorname{im}(\mathcal{S})$ in terms of set equality. Also we note that, by Theorem \ref{thm. RKBS Characterization}, $\mathcal{B} = \operatorname{im}(A)$ and $\sum_{i\in I}^{p}\mathcal{B}_{i} = \operatorname{im}(\mathcal{S})$ as sets.
    Since $\operatorname{im}(A)$ and $\operatorname{im}(\mathcal{S})$ both inherit the same algebraic structure from $\mathbb{R}^{\mathcal{X}}$, we can deduce that they are the same as vector space.
    The only remaining part of the proof is to show that for any $f \in \mathcal{B}$, $\lVert f \rVert_{\mathcal{B}} = \lVert f \rVert_{\sum_{i\in I}^{p}\mathcal{B}_{i}}$.

    To prove (2), suppose that we have the RKBS triple $\mathcal{B}= (\bigoplus_{i\in I}^{p}\Psi_{i},\psi,A)$. We denote $\Phi_{0}:\bigoplus_{i\in I}^{q}\Psi_{i}^{*}\rightarrow \left(\bigoplus_{i\in I}^{p}\Psi_{i}\right)^{*}$ as the isometric isomorphism defined in \eqref{eq. Duality of direct sum}. Since $\psi(x) \in \left(\bigoplus_{i \in I}^{p}\Psi_{i}\right)^{*}$ for all $x\in \mathcal{X}$, we know that 
    \begin{align}
        \Phi_{0}^{-1}\left(\psi(x)\right) \in \bigoplus_{i\in I}^{q}\Psi_{i}^{*}, 
        \quad \lVert \Phi_{0}^{-1}\left(\psi(x)\right)\rVert_{\bigoplus_{i\in I}^{q}\Psi_{i}^{*}} < \infty.
        \label{eq. proof for compatiblity (2)}
    \end{align}
    Now, we define for each $i \in I$, $\psi_{i}: \mathcal{X} \rightarrow \Psi_{i}^{*}$ by $\psi_{i}(x) = p_{i}(\Phi_{0}^{-1}(\psi(x)))$ for $x\in \mathcal{X}$, where $p_{i}$ is $i$-th canonical projection on $\prod_{i\in I}\Psi_{i}^{*}$. Then, for each $i\in I$, there is an RKBS map $A_{i}: \Psi_{i} \rightarrow \mathbb{R}^{\mathcal{X}}$ defined by $(A_{i}\mu_{i})(x)=\left<\psi_{i}(x),\mu_{i}\right>$ for $x\in \mathcal{X}$ and $\mu_{i} \in \Psi_{i}$.
    From Theorem \ref{thm. RKBS Characterization}, we can get a family of RKBS triples 
    $\{\mathcal{B}_{i}=(\Psi_{i},\psi_{i},A_{i})\}_{i\in I}$. 
    Let $\Phi:\bigoplus_{i\in I}^{q}\mathcal{B}_{i}^{*}\rightarrow \left(\bigoplus_{i\in I}^{p}\mathcal{B}_{i}\right)^{*}$ be the isometric isomorphism defined in \eqref{eq. Duality of direct sum}.
    By the above \eqref{eq. proof for compatiblity (2)}, we can deduce that $\left(\psi_{i}(x)\right)_{i\in I} \in \bigoplus_{i\in I}^{q}\Psi_{i}^{*}$.
    Thus, Remark \ref{remark. inf sum condition relate feature map} implies the existence of an RKBS triple for the sum of RKBSs $\sum_{i\in I}^{p}\mathcal{B}_{i} = (\bigoplus_{i\in I}^{p}\mathcal{B}_{i}, \mathbf{s}, \mathcal{S})$.
    From the following series of equations, we can see that $A= \mathcal{S}\circ \widetilde{(\hat{A}_{i})_{i\in I}}\circ\widetilde{(\pi_{i})_{i\in I}}$. 
    For $x \in \mathcal{X}$ and $(\mu_{i})_{i\in I}\in \bigoplus_{i\in I}^{p}\Psi_{i}$, we have that
    \begin{align*}
        \left(A\left((\mu_{i})_{i \in I}\right)\right)(x) 
        &=
        \left<\psi(x), (\mu_{i})_{i \in I}\right>
        =
        \left<\Phi_{0}\left(\Phi_{0}^{-1}(\psi(x))\right), (\mu_{i})_{i \in I}\right>
        =
        \sum_{i \in I}\left<p_{i}(\Phi_{0}^{-1}(\psi(x))),\mu_{i}\right>
        \\
        &=
        \sum_{i\in I}\left<\psi_{i}(x),\mu_{i}\right>
        =
        \left<\Phi((ev_{x}^{i})_{i\in I}),(A_{i}\mu_{i})_{i \in I}\right>
        =
        \left(\mathcal{S}\left((A_{i}\mu_{i})_{i \in I}\right)\right)(x)
        \\
        &=
        \left(\mathcal{S}\left(\widetilde{(A_{i})_{i\in I}}((\mu_{i})_{i \in I})\right)\right)(x)
        =
        \left(\left(\mathcal{S}\circ \widetilde{(\hat{A}_{i})_{i\in I}}\circ\widetilde{(\pi_{i})_{i\in I}}\right)((\mu_{i})_{i \in I})\right)(x).
    \end{align*}
    For similar reasons as the previous case, we only need to prove that for any $f \in \mathcal{B}$, $\lVert f\rVert_{\mathcal{B}} = \lVert f \rVert_{\sum_{i \in I}^{p}\mathcal{B}_{i}}$.

    We start by exploring the definition of each norm.
    The norm on the RKBS $\sum_{i \in I}^{p}\mathcal{B}_{i}$ is given by
    \begin{align*}
        \lVert f \rVert_{\sum_{i \in I}^{p}\mathcal{B}_{i}}^{p}
        &=
        \inf\left\{ \lVert(f_{i})_{i \in I}\rVert_{\bigoplus_{i\in I}^{p}\mathcal{B}_{i}}^{p}: (f_{i})_{i\in I} \in \mathcal{S}^{-1}(f) \right\}
        \\
        &=
        \inf\left\{ \sum_{i \in I}\inf\left\{\lVert \mu_{i}\rVert_{\Psi_{i}}^{p}: \mu_{i} \in A_{i}^{-1}(f_{i})\right\}: (f_{i})_{i\in I} \in \mathcal{S}^{-1}(f) \right\}
    \end{align*}
    
    for $f \in \sum_{i \in I}^{p}\mathcal{B}_{i}$.
    The norm on the RKBS $\mathcal{B}$ is given by
    \begin{align*}
        \lVert f \rVert_{\mathcal{B}}^{p}
        &=
        \inf\left\{ \lVert(\mu_{i})_{i \in I}\rVert_{\bigoplus_{i\in I}^{p}\Psi_{i}}^{p}: (\mu_{i})_{i\in I} \in A^{-1}(f) \right\}
        =
        \inf\left\lVert A^{-1}(f)\right\rVert_{\bigoplus_{i\in I}^{p}\Psi_{i}}^{p}
        \\
        &=
        \inf\left\lVert \widetilde{(A_{i})_{i\in I}}^{-1}\circ \mathcal{S}^{-1}(f)\right\rVert_{\bigoplus_{i\in I}^{p}\Psi_{i}}^{p}
        =
        \inf\left\lVert \bigcup_{(f_{i})_{i\in I}\in \mathcal{S}^{-1}(f)}\widetilde{(A_{i})_{i\in I}}^{-1}\left((f_{i})_{i\in I}\right)\right\rVert_{\bigoplus_{i\in I}^{p}\Psi_{i}}^{p}
        \\
        &=
        \inf\bigcup_{(f_{i})_{i\in I}\in \mathcal{S}^{-1}(f)}\left\lVert \widetilde{(A_{i})_{i\in I}}^{-1}\left((f_{i})_{i\in I}\right)\right\rVert_{\bigoplus_{i\in I}^{p}\Psi_{i}}^{p}
        \displaybreak\\
        &=
        \inf\left\{ \inf\left\lVert \widetilde{(A_{i})_{i\in I}}^{-1}\left((f_{i})_{i\in I}\right)\right\rVert_{\bigoplus_{i\in I}^{p}\Psi_{i}}^{p} : (f_{i})_{i \in I} \in \mathcal{S}^{-1}(f) \right\}
        \\
        &=
        \inf\left\{ \inf\left\{\sum_{i\in I}\lVert\mu_{i}\rVert_{\Psi_{i}}^{p}: (\mu_{i})_{i\in I}\in \widetilde{(A_{i})_{i\in I}}^{-1}\left((f_{i})_{i\in I}\right) \right\}: (f_{i})_{i \in I} \in \mathcal{S}^{-1}(f) \right\}
    \end{align*}
    for $f \in \mathcal{B}$.
    If we denote the set $\left\{\sum_{i\in I}\lVert\mu_{i}\rVert_{\Psi_{i}}^{p}: (\mu_{i})_{i\in I}\in \widetilde{(A_{i})_{i\in I}}^{-1}\left((f_{i})_{i\in I}\right) \right\}$ by $\mathcal{C}$, then we conclude the proof by showing that:
    \begin{align} 
    \sum_{i \in I}\inf\left\{\lVert \mu_{i}\rVert_{\Psi_{i}}^{p}: \mu_{i} \in A_{i}^{-1}(f_{i})\right\}
    =
    \inf\mathcal{C}
    \label{eq. compatiblity norm comparable}
    \end{align}
    for all $(f_{i})_{i \in I} \in \mathcal{S}^{-1}(f)$.
    To show that $\sum_{i \in I}\inf\left\{\lVert \mu_{i}\rVert_{\Psi_{i}}^{p}: \mu_{i} \in A_{i}^{-1}(f_{i})\right\}$ is a lower bound for $\mathcal{C}$, we note that
        $(\mu_{i})_{i\in I} \in \widetilde{(A_{i})_{i\in I}}^{-1}\left((f_{i})_{i\in I}\right)$ is equivalent to
    \begin{align}
        (\mu_{i})_{i\in I} \in \bigoplus_{i \in I}^{p}\Psi_{i}  \text{ and }  \forall i\in I, A_{i}\mu_{i} = f_{i}.
        \label{eq. compatability (3).1}
    \end{align}

    Let $(\nu_{i})_{i\in I} \in \widetilde{(A_{i})_{i\in I}}^{-1}\left((f_{i})_{i\in I}\right)$. Then, by \eqref{eq. compatability (3).1}, we have that
    $\inf\left\{\lVert \mu_{i}\rVert_{\Psi_{i}}^{p}: \mu_{i} \in A_{i}^{-1}(f_{i})\right\} \le \lVert \nu_{i}\rVert_{\Psi_{i}}^{p}$ for all $i \in I$.
    Thus, we deduce that $\sum_{i \in I}\inf\left\{\lVert \mu_{i}\rVert_{\Psi_{i}}^{p}: \mu_{i} \in A_{i}^{-1}(f_{i})\right\} \le \sum_{i \in I}\lVert \nu_{i}\rVert_{\Psi_{i}}^{p}$.
    Now, we have to show that
    $\sum_{i \in I}\inf\left\{\lVert \mu_{i}\rVert_{\Psi_{i}}^{p}: \mu_{i} \in A_{i}^{-1}(f_{i})\right\}$ is the greatest lower bound of $\mathcal{C}$.
    Let $\mathbf{c}$ be an any lower bound of the set $\mathcal{C}$.
    Since we already assumed that $(f_{i})_{i \in I} \in \mathcal{S}^{-1}(f)$, the norm of $(f_{i})_{i\in I}$ in $\bigoplus_{i\in I}^{p}\mathcal{B}_{i}$ is finite. That is, we know that
    $\sum_{i \in I}\inf\left\{\lVert \mu_{i}\rVert_{\Psi_{i}}^{p}: \mu_{i} \in A_{i}^{-1}(f_{i})\right\} = \sum_{i \in I}\lVert f_{i}\rVert_{\mathcal{B}_{i}}< \infty$.
    We denote the set $\left\{ i\in I: \inf\left\{\lVert \mu_{i}\rVert_{\Psi_{i}}^{p}: \mu_{i} \in A_{i}^{-1}(f_{i})\right\}\neq 0 \right\}$ by $H$. Then,
    for $i \in I\setminus H$, $\inf\left\{\lVert \mu_{i}\rVert_{\Psi_{i}}^{p}: \mu_{i} \in A_{i}^{-1}(f_{i})\right\} = 0$.
    Hence, there is a sequence $\{\nu_{i}^{n}\}_{n \in \mathbb{N}} \in A_{i}^{-1}(f_{i})$, so that $\lVert \nu_{i}^{n} - 0 \rVert_{\Psi_{i}}^{p} \rightarrow 0$ as $n \rightarrow \infty$. Furthermore, since $A_{i}^{-1}(f_{i})$ is a translation of $\ker{A_{i}}$, by \eqref{eq. Characterziation thm kerA closed}$, A_{i}^{-1}(f_{i})$ is a closed subset in $\Psi_{i}$. Therefore, we deduce that
    \begin{align}
        0 \in A_{i}^{-1}(f_{i}) \text{ for all }i \in I\setminus H
        \label{eq. compatability (3).2}
    \end{align} For the case of $H$, note that $H$ is a countable subset of $I$. Accordingly, we may take a reordering bijection $g:\mathbb{N} \rightarrow H$.
    By simply using the definition of the infimum, for any $1>\epsilon>0$ and for any $g(n) \in H$, there is a $\nu_{g(n)}\in A_{g(n)}^{-1}(f_{g(n)})$ such that
    \begin{align}
        \inf\left\{ \lVert\mu_{g(n)}\rVert_{\Psi_{g(n)}}^{p}: \mu_{g(n)} \in A_{g(n)}^{-1}(f_{g(n)})\right\} + \frac{1}{4}\cdot\frac{1}{2^{n}}\cdot\epsilon 
        >
        \lVert\nu_{g(n)}\rVert_{\Psi_{g(n)}}^{p}.
        \label{eq. compatability (3).3}
    \end{align}
    Combining the above results, we obtain the following:
    \begin{align*}
        \sum_{i \in I}\inf\left\{\lVert \mu_{i}\rVert_{\Psi_{i}}^{p}: \mu_{i} \in A_{i}^{-1}(f_{i})\right\} + \epsilon
        &>
        \sum_{i \in I}\inf\left\{\lVert \mu_{i}\rVert_{\Psi_{i}}^{p}: \mu_{i} \in A_{i}^{-1}(f_{i})\right\} + \sum_{n=1}^{\infty}\frac{1}{4}\cdot\frac{1}{2^{n}}\cdot\epsilon
        \\
        &=
        \sum_{i \in H}\inf\left\{\lVert \mu_{i}\rVert_{\Psi_{i}}^{p}: \mu_{i} \in A_{i}^{-1}(f_{i})\right\} + \sum_{n=1}^{\infty}\frac{1}{4}\cdot\frac{1}{2^{n}}\cdot\epsilon
        \\
        &= \sum_{n=1}^{\infty}\left(\inf\left\{ \lVert\mu_{g(n)}\rVert_{\Psi_{g(n)}}^{p}: \mu_{g(n)} \in A_{g(n)}^{-1}(f_{g(n)})\right\} + \frac{1}{4}\cdot\frac{1}{2^{n}}\cdot\epsilon\right)
        \\
        &>
        \sum_{n=1}^{\infty}\lVert\nu_{g(n)}\rVert_{\Psi_{g(n)}}^{p}
        =
        \sum_{i \in H}\lVert \nu_{i}\rVert_{\Psi_{i}}^{p}.
    \end{align*}
    Define $\xi_{i}= 
    \begin{cases}
    \nu_{i} & \mbox{if }i\in H, \\
    0 & \mbox{if } i \in I\setminus H
    \end{cases}
    $.
    Then, by \eqref{eq. compatability (3).2} and \eqref{eq. compatability (3).3}, we know that for all $i \in I$, $\xi_{i} \in A_{i}^{-1}(f_{i})$. In addition, from the above inequalities, we also know
    that $\sum_{i\in I}\lVert \xi_{i}\rVert_{\Psi_{i}}^{p}
    \le
    \sum_{i\in I}\lVert f_{i}\rVert_{\mathcal{B}_{i}}^{p} +1 < \infty$. Thus, by \eqref{eq. compatability (3).1}, we deduce that 
    $(\xi_{i})_{i\in I} \in \widetilde{(A_{i})_{i\in I}}^{-1}\left((f_{i})_{i\in I}\right)$ (i.e., $\sum_{i\in I}\lVert \xi_{i}\rVert_{\Psi_{i}}^{p} \in \mathcal{C}$).
    Finally, the following show that $\sum_{i \in I}\inf\left\{\lVert \mu_{i}\rVert_{\Psi_{i}}^{p}: \mu_{i} \in A_{i}^{-1}(f_{i})\right\}$ is the greatest lower bound of $\mathcal{C}$:
    \begin{align*}
        \sum_{i \in I}\inf\left\{\lVert \mu_{i}\rVert_{\Psi_{i}}^{p}: \mu_{i} \in A_{i}^{-1}(f_{i})\right\} + \epsilon
        > 
        \sum_{i \in H}\lVert \nu_{i}\rVert_{\Psi_{i}}^{p}
        =
        \sum_{i\in I}\lVert \xi_{i}\rVert_{\Psi_{i}}^{p}
        \ge \mathbf{c} \quad\text{  for all }1>\epsilon>0,
    \end{align*}
    where $\mathbf{c}$ is a lower bound of the set $\mathcal{C}$.
\end{proof}

\subsection{Decomposition of one-layer neural networks}
The following lemma shows that if there is an isometrically isomorphic feature space, then we can construct an isometrically isomorphic RKBS. Although one can constructively show this by defining a linear map between the quotient spaces induced by a linear map between the original vector spaces (see Exercise 2.4 of \citet{treves2016topological}), the following result readily follows from the properties of RKBS.

\begin{lemma}
    Let $\Psi_{1}$ be a Banach space and let $\mathcal{B}_{2}=(\Psi_{2},\psi_{2},A_{2})$ be an RKBS triple. If there is an isometric isomorphism $\xi: \Psi_{1} \rightarrow \Psi_{2}$, then there is an RKBS triple $\mathcal{B}_{1} = (\Psi_{1},\psi_{1},A_{1})$ such that 
    \begin{enumerate}
        \item $\Psi_{1}/\ker{A_{1}} \underset{\mathcal{B}}{\cong} \Psi_{2}/\ker{A_{2}}$, 
        \item $\mathcal{B}_{1} \underset{\mathcal{B}}{\equiv} \mathcal{B}_{2}$.
    \end{enumerate}
    \label{lem. using lee 4.2.2}
\end{lemma}

\begin{proof}
    Define a feature map $\psi_{1}: \mathcal{X} \rightarrow \Psi_{1}^{*}$ by $\psi_{1}(x) = \psi_{2}(x)\circ \xi$ for $x\in \mathcal{X}$ and a linear map $A_{1}: \Psi_{1} \rightarrow \mathbb{R}^{\mathcal{X}}$ by $\left(A_{1}\mu\right)(x) = <\psi_{1}(x), \mu> $ for $x\in \mathcal{X}$ and $\mu \in \Psi_{1}$. Then, we deduce that $A_{1} = A_{2}\circ \xi$ and $\mathcal{B}_{1} = (\Psi_{1},\psi_{1},A_{1})$ is an RKBS. Thus, we can see that $\mathcal{B}_{1} = \operatorname{im}(A_{1}) = \operatorname{im}(A_{2}\circ\xi)=\operatorname{im}(A_{2})=\mathcal{B}_{2}$.
    For any $f \in \mathcal{B}_{1}$, the following equalities hold:
    \begin{align*}
        \lVert f\rVert_{\mathcal{B}_{1}} 
        &= \inf\left\{\lVert \nu \rVert_{\Psi_{1}}: \nu \in A_{1}^{-1}(f)\right\}
        =\inf\left\{\lVert \nu \rVert_{\Psi_{1}}: \nu \in \xi^{-1}(A_{2}^{-1}(f)) \right\}
        \\
        &= \inf\left\{\lVert \xi(\nu) \rVert_{\Psi_{2}}: \nu \in \xi^{-1}(A_{2}^{-1}(f)) \right\}
        =
        \inf\left\{\lVert \xi(\nu) \rVert_{\Psi_{2}}: \xi(\nu) \in A_{2}^{-1}(f) \right\}
        \\
        &=
        \inf\left\{\lVert \mu \rVert_{\Psi_{2}}: \mu \in A_{2}^{-1}(f) \right\}
        =
        \lVert f \rVert_{\mathcal{B}_{2}}.
    \end{align*}
\end{proof}

Now, we introduce our main theorem. It state that the integral RKBS $\mathcal{F}_{\sigma}(\mathcal{X},\Omega)$ defined in Definition \ref{def. integral RKBS} can be decomposed into the sum of a family of p-norm RKBSs $\{\mathcal{L}_{\sigma}^{1}(\mu_{i})\}_{i\in I}$ defined in Definition \ref{def. p-norm RKBS}, where $\{\mu_{i}\}_{i\in I}$ is a maximal singular family in $P(\Omega)$.
\begin{theorem}
    Let $\{\mu_{i}\}_{i \in I}$ be a maximal singular family in $P(\Omega)$.
    Then, we have the following:
    \begin{align*}
        \mathcal{F}_{\sigma}(\mathcal{X},\Omega) \underset{\mathcal{B}}{\equiv} \sum_{i \in I}\mathcal{L}_{\sigma}^{1}(\mu_{i}).
    \end{align*}
    
    \label{thm. decomposition of the integral RKBS}
\end{theorem}

\begin{proof}
    By Definition \ref{def. integral RKBS}, there exists a map $\psi: \mathcal{X} \rightarrow M(\Omega)^{*}$ defined by $\psi(x) = \Lambda^{*}(\iota(\sigma(x,\cdot)))$ for $x\in \mathcal{X}$. Moreover, there is an RKBS map $A: M(\Omega) \rightarrow \mathbb{R}^{\mathcal{X}}$ defined by $(A(\mu))(x) = <\psi(x),\mu>$ for all $x \in \mathcal{X}$ and $\mu \in M(\Omega)$ such that
    \begin{align*}
        \mathcal{F}_{\sigma}(\mathcal{X},\Omega) \underset{\mathcal{B}}{\cong} M(\Omega)/\ker{A}.
    \end{align*}
    Let $\Theta:\bigoplus_{i\in I}^{1}L^{1}(\mu_{i}) \rightarrow M(\Omega)$ be the isometric isomorphism defined in \eqref{eq. M(K) is vast L^1 sum}.
    Define a map $\overline{\psi}:\mathcal{X} \rightarrow
    \left(\bigoplus_{i\in I}^{1}L^{1}(\mu_{i})\right)^{*}$ by $\overline{\psi}(x) = \psi(x)\circ \Theta$. And consider a map $\overline{A}:\bigoplus_{i\in I}^{1}L^{1}(\mu_{i}) \rightarrow \mathbb{R}^{\mathcal{X}}$ defined by $ \overline{A} = A\circ \Theta$. If we denote  $\overline{\mathcal{B}}$ by $(\bigoplus_{i\in I}^{1}L^{1}(\mu_{i}),\overline{\psi},\overline{A})$, then, by Lemma \ref{lem. using lee 4.2.2}, we have
    \begin{align*}
        M(\Omega)/\ker{A} \underset{\mathcal{B}}{\cong} \bigoplus_{i\in I}^{1}L^{1}(\mu_{i})/\ker{\overline{A}} \quad \text{and} \quad
        \mathcal{F}_{\sigma}(\mathcal{X},\Omega) \underset{\mathcal{B}}{\equiv} \overline{\mathcal{B}}.
    \end{align*}
    Now, let $\Phi_{0}: \bigoplus_{i\in I}^{\infty}\left(L^{1}(\mu_{i})\right)^{*} \rightarrow \left(\bigoplus_{i\in I}^{1}L^{1}(\mu_{i})\right)^{*}$ be the isometric isomorphism defined in \eqref{eq. Duality of direct sum}. For each $i\in I$, if we define a map $\overline{\psi}_{i}:\mathcal{X} \rightarrow \left(L^{1}(\mu_{i})\right)^{*}$ by $\overline{\psi}_{i}(x) = p_{i}\left(\Phi_{0}^{-1}\left(\overline{\psi}(x)\right)\right)$ for $x \in \mathcal{X}$ and define a map $\overline{A_{i}}: L^{1}(\mu_{i}) \rightarrow \mathbb{R}^{\mathcal{X}}$ by $\left(\overline{A_{i}}(h)\right)(x) = \left<\overline{\psi}_{i}(x),h\right>$ for $x\in \mathcal{X}$ and $h \in L^{1}(\mu_{i})$, then by Proposition \ref{prop. compatibility}, we can deduce that
    \begin{align*}
        \bigoplus_{i\in I}^{1}L^{1}(\mu_{i})/\ker{\overline{A}} \underset{\mathcal{B}}{\cong} \sum_{i \in I}\mathcal{B}_{i} \quad \text{and} \quad
        \mathcal{F}_{\sigma}(\mathcal{X},\Omega) \underset{\mathcal{B}}{\equiv} \sum_{i \in I}\mathcal{B}_{i} 
    \end{align*}
    where $\mathcal{B}_{i} = (L^{1}(\mu_{i}), \overline{\psi}_{i},\overline{A}_{i})$ for all $i \in I$. We want to show that $\mathcal{B}_{i}$ is indeed $\mathcal{L}_{\sigma}(\mu_{i})$ for all $i\in I$.
    Suppose for each $i \in I$, $\Xi^{i}:L^{\infty}(\mu_{i}) \rightarrow \left(L^{1}(\mu_{i})\right)^{*}$ is the isometric isomorphism introduced in \eqref{eq. isometric isom dual of L^p and L^q}. 
    According to Definition \ref{def. p-norm RKBS}, it suffices to verify that $\overline{\psi}_{i}(x) = \Xi^{i}(\sigma(x,\cdot))$ for all $x \in \mathcal{X}$ and $i \in I$. This condition is equivalent to $\overline{\psi}(x) 
    =
    \Phi_{0}\left(\left(\Xi^{i}(\sigma(x,\cdot))\right)_{i\in I}\right)$ for all $x\in \mathcal{X}$.
    Hence, we want to prove the following holds: $\left(\Lambda^{*}\left(\iota\left(\sigma(x,\cdot)\right)\right)\circ \Theta\right)\left((f_{i})_{i\in I}\right)
    =
    \Phi_{0}\left((\Xi^{i}(\sigma(x,\cdot)))_{i\in I}\right)\left((f_{i})_{i\in I}\right)$ for all $x\in \mathcal{X}$ and $(f_{i})_{i\in I} \in \bigoplus_{i\in I}^{1}L^{1}(\mu_{i})$.
    First, for the left-hand side, we have:
    \begin{align*}
        &\left(\Lambda^{*}\left(\iota(\sigma(x,\cdot))\right)\circ \Theta\right)\left((f_{i})_{i\in I}\right)
        =
        \left<\Lambda^{*}\left(\iota(\sigma(x,\cdot))\right), _{\mathcal{M}(K)}{\sum}_{i\in I}\rho_{i}\right>
        \\
        &=
        \left<\iota(\sigma(x,\cdot))\circ \Lambda, _{\mathcal{M}(K)}{\sum}_{i\in I}\rho_{i}\right>
        =
        \sum_{i\in I}\left<\iota(\sigma(x,\cdot))\circ \Lambda, \rho_{i}\right>
        \\
        &=
        \sum_{i\in I}\left<\iota(\sigma(x,\cdot)), \Lambda(\rho_{i})\right>
        =
        \sum_{i\in I}\left<\Lambda(\rho_{i}),\sigma(x,\cdot) \right>
        \\
        &=
        \sum_{i\in I}\int_{\Omega}\sigma(x,w)d\rho_{i}(w)
        =
        \sum_{i\in I}\int_{\Omega}\sigma(x,w)f_{i}(w)d\mu_{i}(w).
    \end{align*}
    Next, for the right-hand side, we have: 
    \begin{align*}
        \Phi_{0}\left((\Xi^{i}(\sigma(x,\cdot)))_{i\in I}\right)\left((f_{i})_{i\in I}\right)
        =
        \sum_{i\in I}\left<\Xi^{i}(\sigma(x,\cdot)),f_{i}\right> = \sum_{i\in I}\int_{\Omega}\sigma(x,w)f_{i}(w)d\mu_{i}(w).
    \end{align*}
\end{proof}

\begin{remark}\label{remark.RKBS set equality}
    Since $\Omega$ is a compact metric space in our setting, $\mathcal{L}_{\sigma}^{1}(\mu_{i})$ is a separable RKBS for all $i\in I$. Thus, we decompose the integral RKBS $\mathcal{F}_{\sigma}(\mathcal{X},\Omega)$ into infinitely many separable RKBSs.
\end{remark}

\section{Application}
\subsection{Inclusion relations}

Let $\{\mu_{i}\}_{i\in [n]}$ be any finite family  in $P(\Omega)$. Since for each $i\in [n]$, $k_{i}(x,y)= \int_{\Omega}\sigma(x,w)\sigma(y,w)d\mu_{i}$ for $(x,y)\in \mathcal{X}\times\mathcal{X}$ is the reproducing kernel of $\mathcal{L}_{\sigma}^{2}(\mu_{i})$, the sum kernel $k(x,y)= \sum_{i=1}^{n}k_{i}(x,y)$ is the reproducing kernel of $\sum_{i\in [n]}^{2}\mathcal{L}_{\sigma}^{2}(\mu_{i})$ (The notation $\sum_{i\in [n]}^{2}$ refers to the case where we defined it in Proposition \ref{prop. infinite sum of RKBSs} with 
$p=2$ and 
$I=[n]$). In this setting, we can guarantee that 
$\sum_{i\in [n]}^{2}\mathcal{L}_{\sigma}^{2}(\mu_{i})$ is an RKHS. 
The following proposition shows that when we consider the finite singular family in $P(\Omega)$, the RKHS $\sum_{i\in [n]}^{2}\mathcal{L}_{\sigma}^{2}(\mu_{i})$ contained in the RKBS $\mathcal{F}_{\sigma}(\mathcal{X},\Omega)$ with the same associated function $\sigma$.

\begin{proposition}
    For any finite singular family $\{\mu_{i}\}_{i\in [n]}$ in $P(\Omega)$, we have 
    \begin{align*}
        \sum_{i\in [n]}^{2}\mathcal{L}_{\sigma}^{2}(\mu_{i}) \subset \mathcal{F}_{\sigma}(\mathcal{X},\Omega).    
    \end{align*}
    
    \label{prop. multiple kernel method contained in integral RKBS}
\end{proposition}

\begin{proof}
    As we noted in Definition \ref{def. p-norm RKBS}, we can easily show that for a given $\pi \in P(\Omega)$, $\mathcal{L}_{\sigma}^{2}(\pi) \subset \mathcal{L}_{\sigma}^{1}(\pi)$ and $\lVert f\rVert_{\mathcal{L}_{\sigma}^{1}(\pi)} \le \lVert f\rVert_{\mathcal{L}_{\sigma}^{2}(\pi)}$ for all $f \in \mathcal{L}_{\sigma}^{2}(\pi)$. Let $\{\mu_{i}\}_{i\in I}$ be a maximal singular family containing $\{\mu_{i}\}_{i\in [n]}$.
    Consider a map $\iota: \bigoplus_{i\in [n]}^{2}\mathcal{L}_{\sigma}^{2}(\mu_{i}) \rightarrow \bigoplus_{i\in I}^{1}\mathcal{L}_{\sigma}^{1}(\mu_{i})$ defined by 
    \begin{align*}
        \iota(\mathbf{x})(i) = \begin{cases}
        \mathbf{x}(i), & \mbox{if }i\in [n] \\
        0, & \mbox{if } i \in I\setminus [n]
        \end{cases}
        \quad
        \text{for }\mathbf{x} \in \bigoplus_{i\in [n]}^{2}\mathcal{L}_{\sigma}^{2}(\mu_{i}).
    \end{align*}
    Since $\iota(\mathbf{x})(i) = \mathbf{x}(i) \in \mathcal{L}_{\sigma}^{2}(\mu_{i}) \subset \mathcal{L}_{\sigma}^{1}(\mu_{i})$ for all $i \in [n]$, we know that $\iota(\mathbf{x}) \in \prod_{i\in I}\mathcal{L}_{\sigma}^{1}(\mu_{i})$. Furthermore, by the following inequalities $\sum_{i\in I}\lVert \iota(\mathbf{x})(i)\rVert_{\mathcal{L}_{\sigma}^{1}(\mu_{i})} 
    =
    \sum_{i=1}^{n}\lVert \mathbf{x}(i)\rVert_{\mathcal{L}_{\sigma}^{1}(\mu_{i})} 
    \le
    \sum_{i=1}^{n}\lVert \mathbf{x}(i)\rVert_{\mathcal{L}_{\sigma}^{2}(\mu_{i})} 
    < 
    \infty$, we deduce that $\iota(\mathbf{x}) \in \bigoplus_{i\in I}^{1}\mathcal{L}_{\sigma}^{1}(\mu_{i})$.
    Thus, $\iota$ is well-defined linear map.

    By Remark \ref{remark. inf sum condition relate feature map}, we can define the RKBS linear map for the sum of RKBSs $\mathcal{S}_{1}: \bigoplus_{i\in I}^{1}\mathcal{L}_{\sigma}^{1}(\mu_{i})\rightarrow \mathbb{R}^{\mathcal{X}}$ by $\mathcal{S}_{1}((f_{i})_{i\in I})(x) = \sum_{i\in I}f_{i}(x)$ for $x\in \mathcal{X}$ and $(f_{i})_{i\in I} \in \bigoplus_{i\in I}^{1}\mathcal{L}_{\sigma}^{1}(\mu_{i})$.
    And let $\mathcal{S}_{2}: \bigoplus_{i\in [n]}^{2}\mathcal{L}_{\sigma}^{2}(\mu_{i}) \rightarrow \mathbb{R}^{\mathcal{X}}$ be the RKBS linear map defined by $\mathcal{S}_{2}((f_{i})_{i\in [n]})(x) = \sum_{i\in [n]}f_{i}(x)$ for $x\in \mathcal{X}$ and $(f_{i})_{i\in [n]} \in \bigoplus_{i\in [n]}^{2}\mathcal{L}_{\sigma}^{2}(\mu_{i})$. Now, consider the map $\overline{\iota}: \bigoplus_{i\in [n]}^{2}\mathcal{L}_{\sigma}^{2}(\mu_{i})/\ker{\mathcal{S}_{2}} \rightarrow
    \bigoplus_{i\in I}^{1}\mathcal{L}_{\sigma}^{1}(\mu_{i})/\ker{\mathcal{S}_{1}}$ defined by $\overline{\iota}([\mathbf{x}])= [\iota(\mathbf{x})]$ for $\mathbf{x} \in \bigoplus_{i\in [n]}^{2}\mathcal{L}_{\sigma}^{2}(\mu_{i})$.
    From the following fact
    \begin{align*}
        \iota^{-1}(\ker{\mathcal{S}_{1}}) 
        &=
        \left\{ \mathbf{x}\in \bigoplus_{i\in [n]}^{2}\mathcal{L}_{\sigma}^{2}(\mu_{i}) : \iota(\mathbf{x}) \in \ker{\mathcal{S}_{1}}\right\}
        \\
        &=
        \left\{ \mathbf{x}\in \bigoplus_{i\in [n]}^{2}\mathcal{L}_{\sigma}^{2}(\mu_{i}) : \sum_{i\in I}\left(\iota(\mathbf{x})(i)\right)(x) = 0 \text{ for all }x\in \mathcal{X}\right\}
        = \ker{\mathcal{S}_{2}}
        ,
    \end{align*}
    we deduce that $\overline{\iota}$ is well-defined monomorphism.
    If we consider the map $\widetilde{\operatorname{id}}=\hat{\mathcal{S}_{1}}\circ\overline{\iota}\circ\hat{\mathcal{S}_{2}}^{-1}:\sum_{i\in [n]}^{2}\mathcal{L}_{\sigma}^{2}(\mu_{i}) \rightarrow \sum_{i\in I}\mathcal{L}_{\sigma}^{1}(\mu_{i})$, then it is indeed the identity map. Thus, we have $\sum_{i\in [n]}^{2}\mathcal{L}_{\sigma}^{2}(\mu_{i}) \subset \sum_{i\in I}\mathcal{L}_{\sigma}^{1}(\mu_{i})$. Furthermore, by Theorem \ref{thm. decomposition of the integral RKBS}, we know that $\sum_{i \in I}\mathcal{L}_{\sigma}^{1}(\mu_{i}) = \mathcal{F}_{\sigma}(\mathcal{X},\Omega)$ as a set equality.
\end{proof}

Let a family of continuous functions $\{\sigma_{i}:\mathcal{X}\times\Omega\rightarrow\mathbb{R}\}_{i=1}^{n}$ be given. By the Tietze extension theorem and the pasting lemma, there is a continuous function $\sigma:\mathcal{X}\times\Omega\times[0,1]\rightarrow \mathbb{R}$ which is a extension of the continuous function $\Tilde{\sigma}:\mathcal{X}\times\Omega\times\{\frac{1}{n},\dots,\frac{n-1}{n},1\}\rightarrow \mathbb{R}$ defined by $\Tilde{\sigma}(x,w,\frac{i}{n}) = \sigma_{i}(x,w)$ for all $i=1,\dots, n$, $x\in \mathcal{X}$ and $w\in \Omega$.
In the following proposition, we show that
the finite sum of p-norm RKBSs associated with different functions is contained in the integral RKBS associated with a suitable function when considering a larger parameter space. This means that the class of integral RKBSs is quite large due to its flexibility in choosing the dimension of the parameter space.

\begin{proposition}
    \label{prop. parameter space extension can be embedded in integral RKBS}
    Let a family of continuous functions $\{\sigma_{i}:\mathcal{X}\times\Omega\rightarrow\mathbb{R}\}_{i=1}^{n}$ be given. Let $\{\pi_{i}\}_{i=1}^{n}$ be a collection of probability measures in $\Omega$. Then, we have
    \begin{align*}
        \sum_{i\in [n]}^{2}\mathcal{L}_{\sigma_{i}}^{2}(\Omega,\pi_{i}) \subset \mathcal{F}_{\sigma}(\mathcal{X},\Omega\times[0,1]).
    \end{align*}
\end{proposition}
\begin{proof}
    Fix $i\in [n]$, and define a map $\iota: L^{2}(\Omega,\pi_{i}) \rightarrow L^{2}(\Omega\times[0,1],\pi_{i}\otimes\delta_{i/n})$ by
    \begin{align*}
        \iota(h)(w,r) = \begin{cases}
        h(w), & \mbox{if }r = \frac{i}{n}, \\
        0, & \mbox{otherwise},
    \end{cases}
    \quad
    \text{for }
    w\in \Omega, r\in[0,1],
    \end{align*}
    where $\delta_{i/n}$ is the Dirac measure centred on $i/n$ in $([0,1],\mathcal{B}([0,1]))$.
    Then, $\iota(h)$ is measurable with respect to $(\Omega\times[0,1],\mathcal{B}(\Omega\times[0,1]))$ and $\int_{\Omega\times[0,1]}|\iota(h)(w,r)|^{2}d\pi_{i}\otimes\delta_{i/n} < \infty$. Thus, $\iota$ is well-defined linear map.

    Next, define $A:L^{2}(\Omega,\pi_{i}) \rightarrow \mathbb{R}^{\mathcal{X}}$ by $(Ah)(x)= \int_{\Omega}\sigma_{i}(x,w)h(w)d\pi_{i}$ for $h\in L^{2}(\Omega,\pi_{i})$, $x \in \mathcal{X}$ and define $B:L^{2}(\Omega\times[0,1],\pi_{i}\otimes\delta_{i/n})\rightarrow\mathbb{R}^{\mathcal{X}}$ by $(B\Tilde{h})(x) = \int_{\Omega\times[0,1]}\sigma(x,w,r)\Tilde{h}(w,r)d\pi_{i}\otimes\delta_{i/n}$ for $\Tilde{h} \in L^{2}(\Omega\times[0,1],\pi_{i}\otimes\delta_{i/n})$, $x\in \mathcal{X}$, which are the RKBS linear maps introduced in Definition \ref{def. p-norm RKBS}. From the following equations
    \begin{align*}
        \iota^{-1}(\ker{B})
        &=
        \left\{h \in L^{2}(\Omega,\pi_{i}): \iota(h) \in \ker{B}\right\}
        =
        \left\{h \in L^{2}(\Omega,\pi_{i}): B(\iota(h))=0 \right\}
        \\
        &=
        \left\{h \in L^{2}(\Omega,\pi_{i}): \int_{\Omega}\int_{[0,1]}\sigma(x,w,r)\iota(h)(w,r)d\delta_{i/n}d\pi_{i}\right\}
        \\
        &=
        \left\{h \in L^{2}(\Omega,\pi_{i}): \int_{\Omega}\sigma_{i}(x,w)h(w)d\pi_{i}\right\}
        =
        \ker{A},
    \end{align*}
    it follows that $\overline{\iota}:L^{2}(\Omega,\pi_{i})\rightarrow L^{2}(\Omega\times[0,1],\pi_{i}\otimes\delta_{i/n})$ defined by $\overline{\iota}([h])=[\iota(h)]$ for $h\in L^{2}(\Omega,\pi_{i})$ is well-defined monomorphism. Now, consider a map $\Tilde{\operatorname{id}}=\hat{B}\circ\overline{\iota}\circ\hat{A}^{-1}:\mathcal{L}_{\sigma_{i}}^{2}(\Omega,\pi_{i})\rightarrow\mathcal{L}_{\sigma}^{2}(\Omega\times[0,1],\pi_{i}\otimes\delta_{i/n})$. For any $Ah \in \mathcal{L}_{\sigma_{i}}^{2}(\Omega,\pi_{i}) = \operatorname{im}(A)$, we have $\Tilde{\operatorname{id}}(Ah)= \Tilde{B}\circ\overline{\iota}([h])=\hat{B}([\iota(h)])=B\iota(h)=Ah$, which shows that $\Tilde{\operatorname{id}}$ is the identity map on $\mathcal{L}_{\sigma_{i}}^{2}(\Omega,\pi_{i})$.
    Furthermore, we have
    \begin{align*}
        \lVert Ah\rVert_{\mathcal{L}_{\sigma}^{2}(\Omega\times[0,1],\pi_{i}\otimes\delta_{i/n})}
        &=
        \lVert B\iota(h)\rVert_{\mathcal{L}_{\sigma}^{2}(\Omega\times[0,1],\pi_{i}\otimes\delta_{i/n})}
        \\
        &=
        \lVert [\iota(h)]\rVert_{L_{\sigma}^{2}(\Omega\times[0,1],\pi_{i}\otimes\delta_{i/n})/\ker{B}}
        \\
        &=
        \inf_{\Tilde{g}\in \ker{B}}\lVert \iota(h) + \Tilde{g}\rVert_{L_{\sigma}^{2}(\Omega\times[0,1],\pi_{i}\otimes\delta_{i/n})}
        \\
        &\le
        \inf_{g\in \ker{A}}\lVert \iota(h) + \iota(g)\rVert_{L_{\sigma}^{2}(\Omega\times[0,1],\pi_{i}\otimes\delta_{i/n})}
        \\
        &=
        \lVert Ah\rVert_{\mathcal{L}_{\sigma_{i}}^{2}(\Omega,\pi_{i})}.
    \end{align*}
    Thus, for $i=1,\dots,n$, we have $ \mathcal{L}_{\sigma_{i}}^{2}(\Omega,\pi_{i}) \subset \mathcal{L}_{\sigma}^{2}(\Omega\times[0,1],\pi_{i}\otimes\delta_{i/n})$
    and for all $f \in \mathcal{L}_{\sigma}^{2}(\Omega,\pi_{i})$, it holds that $ \lVert f\rVert_{\mathcal{L}_{\sigma}^{2}(\Omega\times[0,1],\pi_{i}\otimes\delta_{i/n})}
    \le
    \lVert f\rVert_{\mathcal{L}_{\sigma_{i}}^{2}(\Omega,\pi_{i})}.$
    From this, we deduce that $\sum_{i\in [n]}^{2}\mathcal{L}_{\sigma_{i}}^{2}(\Omega,\pi_{i})\subset \sum_{i\in [n]}^{2}\mathcal{L}_{\sigma}^{2}(\Omega\times[0,1],\pi_{i}\otimes\delta_{i/n})$.
    Since the family $\{\pi_{i}\otimes\delta_{i/n}\}_{i=1}^{n}$ is singular in $P(\Omega\times[0,1])$, it follows from Proposition \ref{prop. multiple kernel method contained in integral RKBS} that $\sum_{i\in [n]}^{2}\mathcal{L}_{\sigma_{i}}^{2}(\Omega,\pi_{i}) \subset \mathcal{F}_{\sigma}(\mathcal{X},\Omega\times[0,1]).$
\end{proof}

\begin{remark}
    For the purpose of a realistic application, we consider the case where $p=2$ in this section, but the results can also be generalized to the case where $1\le p <\infty$. Note that from Corollary 13 in \citet{spek2022duality}, it is known that $\mathcal{F}_{\sigma}(\mathcal{X},\Omega)=\bigcup_{\pi \in P(\Omega)}\mathcal{L}_{\sigma}^{p}(\pi)$. Thus, Proposition \ref{prop. multiple kernel method contained in integral RKBS} can be obtained without needing to consider the infinite sum of RKBSs. However, using this approach allows for a more systematic exploration.
\end{remark}

\subsection{Existence of general solutions in the integral RKBS class}
Following \citet{steinwart2008support}, we introduce the concepts related to the integral form of the L-risk to be used in our framework.

\begin{definition}[\citep{steinwart2008support}]
    Let $(\mathcal{X},\mathcal{A})$ be a measurable space and $\mathcal{Y} \subset \mathbb{R}$ be a closed subset. Then a function $L: \mathcal{X} \times \mathcal{Y} \times \mathbb{R} \rightarrow [0,\infty)$ is called a loss function if it is measurable.
    Furthermore, we say a loss function  $L:\mathcal{X}\times \mathcal{Y}\times \mathbb{R} \rightarrow [0,\infty)$ is continuous if $L(x,y,\cdot):\mathbb{R}\rightarrow [0,\infty)$ is continuous for all $x\in \mathcal{X}, y\in \mathcal{Y}$.
\end{definition}
\begin{definition}[\citep{steinwart2008support}]
    Let $L: \mathcal{X} \times \mathcal{Y} \times \mathbb{R} \rightarrow [0,\infty)$ be a loss funciton and $P$ be a probability measure on $\mathcal{X} \times \mathcal{Y}$. Then, for a measurable function $f:\mathcal{X} \rightarrow \mathbb{R}$, the L-risk is defined by 
    \begin{align*}
        \mathcal{R}_{L,P}(f) =\int_{\mathcal{X}\times \mathcal{Y}}L(x,y,f(x))dP(x,y) = \int_{\mathcal{X}}\int_{\mathcal{Y}}L(x,y,f(x))dP(y|x)dP_{\mathcal{X}}(x).
    \end{align*}
\end{definition}

Since $(x,y)\mapsto (x,y,f(x))$ is measurable and the definition of loss function, the L-risk map $\mathcal{R}_{L,P}: \mathfrak{L}^{0}(X)\rightarrow [0,\infty]$ is well-defined. Furthermore, since $\mathcal{Y}$ is a closed subset of $\mathbb{R}$, the existence of the regular conditional probability $P(\cdot|x)$ and the marginal distribution $P_{\mathcal{X}}$ of $P$ on $\mathcal{X}$ follows. As a special case, given a sequence of data $(x_{i},y_{i})_{i=1}^{n} \in (\mathcal{X}\times \mathcal{Y})^{n}$, we consider the empirical measure  $D:= \frac{1}{n}\sum_{i\in [n]}\delta_{(x_{i},y_{i})}$ defined on $\mathcal{X}\times \mathcal{Y}$, where $\delta_{(x_{i},y_{i})}$ is the Dirac measure centered at $(x_{i},y_{i})$. Using this, we obtain the empirical L-risk defined as follows:
\begin{align*}
    R_{L,D}(f) := \frac{1}{n}\sum_{i=1}^{n}L(x_{i},y_{i},f(x_{i})).
\end{align*}

Let $L: \mathcal{X} \times \mathcal{Y} \times \mathbb{R} \rightarrow [0,\infty)$ be a loss function, $\mathcal{F}$ an arbitrary function space, and $P$ a distribution on $\mathcal{X}\times\mathcal{Y}$. For any $\lambda>0$, we aim to solve the following regularized optimization problem:
\begin{align*}
    \lambda\lVert f_{P,\lambda}\rVert_{\mathcal{F}} + \mathcal{R}_{L,P}(f_{P,\lambda}) = \inf\left\{\lambda\lVert f\rVert_{\mathcal{F}} + \mathcal{R}_{L,P}(f): f \in \mathcal{F}\right\}.
\end{align*}
A solution $f_{P,\lambda}$ that satisfies this condition is called a general solution. In practice, since the data distribution $P$ is unknown, we instead solve an optimization problem based on the empirical L-risk:
\begin{align*}
    \lambda\lVert f_{D,\lambda}\rVert_{\mathcal{F}} + \mathcal{R}_{L,D}(f_{D,\lambda}) = \inf\left\{\lambda\lVert f\rVert_{\mathcal{F}} + \mathcal{R}_{L,D}(f): f \in \mathcal{F}\right\}.
\end{align*}
This approach is referred to as the ERM (Empirical Risk Minimization) framework.

In general, the L-risk does not necessarily have to be finite, but the L-risk is finite for a loss function that satisfies the following definition.

\begin{definition}[\citep{steinwart2008support}]
    Let $P$ be a distribution on $\mathcal{X}\times \mathcal{Y}$.
    A loss function $L:\mathcal{X}\times\mathcal{Y}\times\mathbb{R}\rightarrow [0,\infty)$ is called P-integrable Nemitski loss if there is a measurable function $b:\mathcal{X}\times\mathcal{Y}\rightarrow [0,\infty)$ with $b\in \mathfrak{L}^{1}(P)$ and an increasing function $h:[0,\infty)\rightarrow[0,\infty)$ such that
    \begin{align*}
        L(x,y,t) \le b(x,y) + h(\lvert t \rvert)
    \end{align*}
    for all $(x,y,t) \in \mathcal{X}\times\mathcal{Y}\times \mathbb{R}$. In particular, for each $f\in \mathfrak{L}^{\infty}(P_{\mathcal{X}})$, we have $\mathcal{R}_{L,P}(f) < \infty$.
\end{definition}

The following result, originally presented in \citep{spek2022duality}, demonstrates the existence of a pre-dual space of the integral RKBS $\mathcal{F}_{\sigma}(\mathcal{X},\Omega)$, which provides a key idea for proving the existence of general solutions to one-layer neural networks.

\begin{proposition}[Pre-dual of integral RKBS \citep{spek2022duality}]
    Consider the operator $A^{*}:\mathcal{M}(\mathcal{X}) \rightarrow C(\Omega)$ defined by 
    \begin{align*}
        A^{*}\rho(w) = \int_{\mathcal{X}}\sigma(x,w)d\rho(x) \text{ for all } w\in \Omega,  \rho \in \mathcal{M}(\mathcal{X}).
    \end{align*}
    Define the normed vector space
    \begin{align*}
        \mathcal{G}_{\sigma}(\Omega,\mathcal{X}) =\{g \in C(\Omega): \exists \rho \in \mathcal{M}(\mathcal{X}) \text{ s.t. } A^{*}\rho = g\}.
    \end{align*}
    Then, $\mathcal{G}_{\sigma}(\Omega,\mathcal{X})$ serves as pre-dual of integral RKBS $\mathcal{F}_{\sigma}(\mathcal{X},\Omega)$, with the duality pairing given by
    \begin{align*}
        \left<f,g\right> = \int_{\mathcal{X}\times \Omega}\sigma(x,w)d(\rho \otimes \mu)    
    \end{align*}
    for all $f \in \mathcal{F}_{\sigma}(\mathcal{X},\Omega), g \in \mathcal{G}_{\sigma}(\Omega,\mathcal{X})$ where $\rho \in \mathcal{M}(\mathcal{X}), \mu \in \mathcal{M}(\Omega)$ such that $\rho \in (A^{*})^{-1}(g), \mu \in A^{-1}(f)$.
    \label{prop. duality of integral RKBS}
\end{proposition}
Before proceeding, we make two observations.
First, by Proposition \ref{prop. A is compact operator}, we observe that, in general, $\mathcal{G}_{\sigma}(\Omega, \mathcal{X}) $ is not a Banach space. Second, when considering the weak* topology on the integral RKBS $ (\mathcal{F}_{\sigma}(\mathcal{X},\Omega), w^*) $, all evaluation functionals $ev_{x}: (\mathcal{F}_{\sigma}(\mathcal{X},\Omega), w^*) \to \mathbb{R}$  
are continuous. More precisely, suppose that a net $\{f_a\} $ converges to $ f $ in the weak* topology, that is,  
$ \langle f, g \rangle = \lim_{a} \langle f_a, g \rangle \text{ for all } g \in \mathcal{G}_{\sigma}(\Omega, \mathcal{X}).$  
For any \( f \in \mathcal{F}_{\sigma}(\mathcal{X},\Omega) \), we have $\langle ev_{x_0}, f \rangle = \langle f, g_0 \rangle$,
where \( g_0 \) is defined as
$ g_0 = \int_{\mathcal{X}} \sigma(x, \cdot) \, d\delta_{x_0}(x)$.  
Taking the limits, it follows that
\[
\lim_{a} \langle ev_{x_0}, f_a \rangle = \lim_{a} \langle f_a, g_0 \rangle = \langle f, g_0 \rangle = f(x_0) \quad \text{for all } x_0 \in \mathcal{X}.
\]  

The following theorem establishes the existence of general solutions for one-layer neural networks. This result is inspired by the approach used in \citet{steinwart2008support} to prove the existence of general solutions for support vector machines (SVM). However, since the integral RKBS does not allow for a direct use of the dual or bidual space as in Hilbert spaces or reflexive Banach spaces, a theorem such as proposition 6 on p. 75 in \citet{ekeland1983infinite} cannot be applied directly. Fortunately, by utilizing Proposition \ref{prop. duality of integral RKBS}, we were able to exploit the existence of a pre-dual space and apply the weak* topology to overcome this limitation. 

\begin{theorem}[Existence of general solution of one-layer neural networks]
    Let an integral RKBS $\mathcal{F}_{\sigma}(\mathcal{X},\Omega)$ be given.
    Let $L:\mathcal{X}\times \mathcal{Y}\times \mathbb{R} \rightarrow [0,\infty)$ be a continuous loss function and $P$ be a ditribution on $\mathcal{X} \times \mathcal{Y}$ such that $L$ is a $P$-integrable Nemitski loss. Then, for any $\lambda>0$, there is a $f_{P,\lambda} \in \mathcal{F}_{\sigma}(\mathcal{X},\Omega)$ such that
    \begin{align*}
    \lambda\lVert f_{P,\lambda}\rVert_{\mathcal{F}_{\sigma}(\mathcal{X},\Omega)} + \mathcal{R}_{L,P}(f_{P,\lambda}) = \inf\left\{\lambda\lVert f\rVert_{\mathcal{F}_{\sigma}(\mathcal{X},\Omega)} + \mathcal{R}_{L,P}(f): f \in \mathcal{F}_{\sigma}(\mathcal{X},\Omega)\right\}.
    \end{align*}
    \label{thm. general solution one-layer nn}
\end{theorem}

\begin{proof} 
    Since $\mathcal{F}_{\sigma}(\mathcal{X},\Omega) \subsetneq C(X) \subset \mathfrak{L}^{\infty}(P_{X})$ and $R_{L,P}(f) < \infty$ for all $f \in \mathcal{F}_{\sigma}(\mathcal{X},\Omega)$, the L-risk map $\mathcal{R}_{L,P}:\mathcal{F}_{\sigma}(\mathcal{X},\Omega) \rightarrow \mathbb{R}$ is well defined. By Proposition \ref{prop. duality of integral RKBS}, we can cosider the weak* topology of $\mathcal{F}_{\sigma}(\mathcal{X},\Omega)$. We want to show that the L-risk map $\mathcal{R}_{L,P}: (\mathcal{F}_{\sigma}(\mathcal{X},\Omega),w^{*}) \rightarrow \mathbb{R}$ is weak* lower semi-continuous.
    As mentioned earlier, all evaluation functionals $ev_{x}:(\mathcal{F}_{\sigma}(\mathcal{X},\Omega),w^{*}) \rightarrow \mathbb{R}$ are continuous.
    Thus, if $f_{a} \rightarrow f$ in $(\mathcal{F}_{\sigma}(\mathcal{X},\Omega),w^{*})$, then $f_{a}(x) \rightarrow f(x)$ for all $x \in \mathcal{X}$.
    Due to the continuity of the loss funciton $L$, we obtain the following:
    \begin{align*}
        \mathcal{R}_{L,P}(f) &= \int_{X\times Y}L(x,y,f(x))dP(x,y) 
        =
        \int_{X\times Y}\lim_{a} L(x,y,f_{a}(x))dP(x,y)
        \\
        &\le \lim\inf_a\int_{X\times Y}L(x,y,f_{a}(x))dP(x,y) = \lim\inf_{a}\mathcal{R}_{L,P}(f_{a}). 
    \end{align*}
    Consider the set $C$ defined by
    \begin{align*}
        C := \{f \in \mathcal{F}_{\sigma}(\mathcal{X},\Omega) : \lambda\lVert f \rVert_{\mathcal{F}_{\sigma}(\mathcal{X},\Omega)} + \mathcal{R}_{L,P}(f) \le M\}.
    \end{align*}
    where $\infty > M := \mathcal{R}_{L,P}(0) \ge 0$. Then, we have  $0 \in C$ and $C \subseteq \frac{M}{\lambda}B_{\mathcal{F}_{\sigma}(\mathcal{X},\Omega)}$ where $B_{\mathcal{F}_{\sigma}(\mathcal{X},\Omega)}$ is ball in $\mathcal{F}_{\sigma}(\mathcal{X},\Omega)$.
    Since $\lVert\cdot\rVert_{\mathcal{F}_{\sigma}(\mathcal{X},\Omega)}:(\mathcal{F}_{\sigma}(\mathcal{X},\Omega),w^{*}) \rightarrow \mathbb{R}$ is weak* lower semi-continuous 
    (see \uppercase\expandafter{\romannumeral5} \S1 Exercise 9 of \citet{conway1997course}), we can deduce that $f\mapsto \lambda\lVert f\rVert_{\mathcal{F}_{\sigma}(\mathcal{X},\Omega)}+\mathcal{R}_{L,P}(f)$ is weak* lower semi-continuous. Thus, $C$ is closed set in $(\mathcal{F}_{\sigma}(\mathcal{X},\Omega),w^{*})$. From the Banach-Aloaglu theorem, we can infer that $C$ is compact set in $(\mathcal{F}_{\sigma}(\mathcal{X},\Omega),w^{*})$. Thus, $f\mapsto \lambda\lVert f\rVert_{\mathcal{F}_{\sigma}(\mathcal{X},\Omega)}+\mathcal{R}_{L,P}(f)$ attains its minimum on $C$ and hence on $\mathcal{F}_{\sigma}(\mathcal{X},\Omega)$.
\end{proof}

\subsection{Reformulation scheme}

We examine how the existence of a solution in the feature space (a general Banach space) implies the existence of a corresponding solution in the hypothesis space (RKBS). Roughly speaking, an optimization problem in the hypothesis space can be reformulated as an optimization problem in the feature space. This result was established for integral RKBSs in \cite{bartolucci2023understanding}. For the reader's convenience, we provide a slightly modified proof. Moreover, for generality, we extend our discussion beyond empirical L-risk to general L-risk and beyond integral RKBSs to general RKBSs.
 
\begin{lemma}
    \label{lem. isometric v.s. have same minimizer}
    Let $V,W$ be normed vector spaces and $T:V\rightarrow W$ be an isometric isomorphism.
    For any functional $J: W\rightarrow \mathbb{R}\cup\{+\infty\}$, we have 
    \begin{align*}
        \inf\left\{J(\alpha)+\lVert \alpha\rVert_{W}:\alpha \in W\right\}
        =
        \inf\left\{J(T\theta)+\lVert \theta\rVert_{V}:\theta \in V\right\}.
    \end{align*}
\end{lemma}

\begin{proof}
    $\inf\left\{J(\alpha)+\lVert \alpha\rVert_{W}:\alpha \in W\right\}
    =
    \inf\left\{J(T\theta)+\lVert T\theta\rVert_{W}:\alpha \in W\right\}
    =
        \inf\left\{J(T\theta)+\lVert \theta\rVert_{V}:\alpha \in V\right\}.$ 
\end{proof}

\begin{proposition}[\citep{bartolucci2023understanding}]
    Let $\mathcal{B}=(\Psi,\psi,A)$ be an RKBS such that $\mathcal{B}\subset \mathfrak{L}^{0}(\mathcal{X})$.
    \\
    Then, we have
    \begin{align*}
        \underbrace{\inf\left\{\mathcal{R}_{L,P}(f)+\lVert f\rVert_{\mathcal{B}}: f\in \mathcal{B}\right\}}_{(1)} 
        =
        \underbrace{
        \inf\left\{\mathcal{R}_{L,P}(A\mu)+\lVert \mu\rVert_{\Psi}: \mu \in \Psi\right\}}_{(2)}. 
    \end{align*}
    Furthermore, if $\mu^{*}$ is a minimizer of (2), then $f^{*}:=A\mu^{*}$ is a minimizer of (1).
    \label{prop. reformulation of optimization problem}
\end{proposition}
\begin{proof}
    Note that $\hat{A}:\Psi/\ker{A} \rightarrow \mathcal{B}$ is an isometric isomorphism. By Lemma \ref{lem. isometric v.s. have same minimizer}, we have the followings:
    \begin{align*}
        \inf\left\{\mathcal{R}_{L,P}(f)+\lVert f\rVert_{\mathcal{B}}: f\in \mathcal{B}\right\}
        &=
        \inf\left\{\mathcal{R}_{L,P}(\hat{A}[\mu])+\lVert [\mu]\rVert_{\Psi/\ker{A}}: [\mu] \in \Psi/\ker{A}\right\}
        \\
        &=
        \inf\left\{\inf\left\{\mathcal{R}_{L,P}(\hat{A}[\mu])+\lVert \nu \rVert_{\Psi}: \nu \in [\mu]\right\}: [\mu] \in \Psi/\ker{A}\right\}
    \end{align*}
    If we denote $C_{[\mu]}:
    =
    \left\{\mathcal{R}_{L,P}(\hat{A}[\mu])+\lVert \nu \rVert_{\Psi}:\nu \in [\mu]\right\}
    =
    \mathcal{R}_{L,P}(\hat{A}[\mu])+\lVert [\mu] \rVert_{\Psi}$, then we have:
    \begin{align*}
        \inf\left\{\mathcal{R}_{L,P}(f)+\lVert f\rVert_{\mathcal{B}}: f\in \mathcal{B}\right\} 
        =
        \inf\left\{\inf C_{[\mu]}: [\mu] \in \Psi/\ker{A}\right\}
    \end{align*}
    From the fact that $\bigcup_{[\mu]\in \Psi/\ker{A}}[\mu] = \Psi$, we have $\bigcup_{[\mu]\in \Psi/\ker{A}}C_{[\mu]} 
    =
    \left\{\mathcal{R}_{L,P}(A\mu)+\lVert \mu\rVert_{\Psi}: \mu \in \Psi\right\}$. Thus, we can infer that
    \begin{align*}
        &\inf\left\{\mathcal{R}_{L,P}(f)+\lVert f\rVert_{\mathcal{B}}: f\in \mathcal{B}\right\} 
        =
        \inf\left\{\inf C_{[\mu]}: [\mu] \in \Psi/\ker{A}\right\}
        \\
        &=
        \inf \bigcup_{[\mu]\in \Psi/\ker{A}}C_{[\mu]}
        =
        \inf\left\{\mathcal{R}_{L,P}(A\mu)+\lVert \mu\rVert_{\Psi}: \mu \in \Psi\right\}.
    \end{align*}
    Now, let $\mu^{*}$ be a minimizer of (2). Then, for each $\nu \in \Psi$, $\mathcal{R}_{L,P}(A\mu^{*})+\lVert \mu^{*}\rVert_{\Psi} \le  \mathcal{R}_{L,P}(A\nu)+\lVert \nu\rVert_{\Psi}$. If we denote $f = A\nu$, then we have
    \begin{align*}
        \mathcal{R}_{L,P}(A\mu^{*})+\lVert \mu^{*}\rVert_{\Psi} 
        &\le
        \inf\left\{\mathcal{R}_{L,P}(A\nu)+\lVert \nu\rVert_{\Psi}: \nu \in A^{-1}(f)\right\}
        \\
        &=
        \mathcal{R}_{L,P}(A\nu) + \inf\left\{\lVert \nu\rVert_{\Psi}: \nu \in A^{-1}(f)\right\}
        \\
        &=
        \mathcal{R}_{L,P}(f) + \lVert f \rVert_{\mathcal{B}}
    \end{align*}
    In other words, if we denote $f^{*}= A\mu^{*}$, then for any $f\in \mathcal{B}$, we have 
    \begin{align*}
        \mathcal{R}_{L,P}(f^{*}) + \lVert f^{*} \rVert_{\mathcal{B}} 
        \le
        \mathcal{R}_{L,P}(A\mu^{*}) + \lVert \mu^{*} \rVert_{\mathcal{B}}
        \le
        \mathcal{R}_{L,P}(f) + \lVert f \rVert_{\mathcal{B}}.
    \end{align*}
\end{proof}

The following result generalizes Proposition \ref{prop. reformulation of optimization problem}. It shows that when the feature space is a direct sum of Banach spaces, the optimization problem in the feature space can be reformulated as an optimization problem in the direct sum of hypothesis spaces (RKBSs). While the proof follows a similar approach, it relies on the operators constructed in Proposition \ref{prop. compatibility}. For clarification, refer to the proof of Proposition \ref{prop. compatibility}.

\begin{proposition}
    For $1\le p < \infty$,
    let RKBS triple $\mathcal{B}=(\bigoplus_{i\in I}^{p}\Psi_{i},\psi,A)$ with $\mathcal{B} \subset \mathfrak{L}^{0}(\mathcal{X})$ be given. Then, by Proposition \ref{prop. compatibility}, there is a family of RKBS triples $\{\mathcal{B}_{i} = (\Psi_{i},\psi_{i},A_{i})\}_{i\in I}$ such that $\mathcal{B} \underset{\mathcal{B}}{\equiv} \sum_{i\in I }^{p}\mathcal{B}_{i} = (\bigoplus_{i\in I}^{p}\mathcal{B}_{i},\mathbf{s},\mathcal{S})$. Under this assumption, we obtain the following result:
    \begin{align*}
        &\underbrace{\inf\left\{\mathcal{R}_{L,P}(A(\mu_{i})_{i\in I})+\lVert (\mu_{i})_{i\in I}\rVert_{\bigoplus_{i\in I}^{p}\Psi_{i}}: (\mu_{i})_{i\in I}\in \bigoplus_{i\in I}^{p}\Psi_{i}\right\}}_{(1)} 
        \\
        &=
        \underbrace{\inf\left\{\mathcal{R}_{L,P}(\mathcal{S}(f_{i})_{i\in I})+\lVert (f_{i})_{i\in I}\rVert_{\bigoplus_{i\in I}^{p}\mathcal{B}_{i}}: (f_{i})_{i\in I}\in \bigoplus_{i\in I}^{p}\mathcal{B}_{i}\right\}}_{(2)}
    \end{align*}
    Furthermore, if $(\mu^{*}_{i})_{i\in I}$ is a minimizer of (1), then $(f^{*}_{i})_{i\in I}:=(A_{i}\mu^{*}_{i})_{i\in I}$ is a minimizer of (2).
    \label{prop. reformulate optimization problem of direct sum space}
\end{proposition}
\begin{proof}
    First, by \eqref{eq. compatiblity norm comparable}, we have
    \begin{align*}
        \lVert (f_{i})_{i\in I}\rVert_{\bigoplus_{i\in I}^{p}\mathcal{B}_{i}}^{p} = \sum_{i\in I}\inf\left\{\lVert \mu_{i}\rVert_{\Psi_{i}}^{p}: \mu_{i} \in A_{i}^{-1}(f_{i})\right\}
        = \inf\left\{ \lVert (\mu_{i})_{i\in I}\rVert_{\bigoplus_{i\in I}^{p}\Psi_{i}}^{p} : (\mu_{i})_{i\in I} \in \widetilde{(A_{i})_{i\in I}}^{-1}((f_{i})_{i\in I})  \right\}.
    \end{align*}
    Thus, the optimization problem (2) can be expressed as follows:
    \begin{align*}
        (2)
        &=
        \inf\left\{\mathcal{R}_{L,P}(\mathcal{S}(f_{i})_{i\in I})+\lVert (f_{i})_{i\in I}\rVert_{\bigoplus_{i\in I}^{p}\mathcal{B}_{i}}: (f_{i})_{i\in I}\in \bigoplus_{i\in I}^{p}\mathcal{B}_{i}\right\}
        \\
        &=
        \inf\left\{\mathcal{R}_{L,P}(\mathcal{S}(f_{i})_{i\in I})+
        \inf\left\{ \lVert (\mu_{i})_{i\in I}\rVert_{\bigoplus_{i\in I}^{p}\Psi_{i}} : (\mu_{i})_{i\in I} \in \widetilde{(A_{i})_{i\in I}}^{-1}((f_{i})_{i\in I})  \right\}
        : (f_{i})_{i\in I}\in \bigoplus_{i\in I}^{p}\mathcal{B}_{i}\right\}
        \\
        &=
        \inf\left\{
        \inf\left\{ \mathcal{R}_{L,P}(\mathcal{S}(f_{i})_{i\in I})
        +
        \lVert (\mu_{i})_{i\in I}\rVert_{\bigoplus_{i\in I}^{p}\Psi_{i}} : (\mu_{i})_{i\in I} \in \widetilde{(A_{i})_{i\in I}}^{-1}((f_{i})_{i\in I})  \right\}
        : (f_{i})_{i\in I}\in \bigoplus_{i\in I}^{p}\mathcal{B}_{i}
        \right\}
        \\
        &=
        \inf\left\{
        \inf\left\{ \mathcal{R}_{L,P}\left(\mathcal{S}\left(\widetilde{(A_{i})_{i\in I}}(\mu_{i})_{i\in I}\right)\right)
        +
        \lVert (\mu_{i})_{i\in I}\rVert_{\bigoplus_{i\in I}^{p}\Psi_{i}} : (\mu_{i})_{i\in I} \in \widetilde{(A_{i})_{i\in I}}^{-1}((f_{i})_{i\in I})  \right\}
        : (f_{i})_{i\in I}\in \bigoplus_{i\in I}^{p}\mathcal{B}_{i}
        \right\}
    \end{align*}
    Since $\widetilde{(A_{i})_{i\in I}}$ is surjective, we have
    \begin{align*}
        \bigcup_{(f_{i})_{i\in I} \in \bigoplus_{i\in I}^{p}\mathcal{B}_{i}}\widetilde{(A_{i})_{i\in I}}^{-1}((f_{i})_{i\in I}) = \bigoplus_{i\in I}^{p}\Psi_{i}.
    \end{align*}
    For convenience, define for each $(f_{i})_{i\in I} \in \bigoplus_{i\in I}^{p}\mathcal{B}_{i}$
    \begin{align*}
        C_{\widetilde{(A_{i})_{i\in I}}^{-1}((f_{i})_{i\in I})}
        :=
        \left\{ \mathcal{R}_{L,P}\left(\mathcal{S}\left(\widetilde{(A_{i})_{i\in I}}(\mu_{i})_{i\in I}\right)\right)
        +
        \lVert (\mu_{i})_{i\in I}\rVert_{\bigoplus_{i\in I}^{p}\Psi_{i}} : (\mu_{i})_{i\in I} \in \widetilde{(A_{i})_{i\in I}}^{-1}((f_{i})_{i\in I})  \right\},
    \end{align*}
    so that
    \begin{align*}
        \underset{(f_{i})_{i\in I} \in \bigoplus_{i\in I}^{p}\mathcal{B}_{i}
        }{\bigcup}C_{\widetilde{(A_{i})_{i\in I}}^{-1}((f_{i})_{i\in I})}
        =
        \left\{ \mathcal{R}_{L,P}\left(\mathcal{S}\left(\widetilde{(A_{i})_{i\in I}}(\mu_{i})_{i\in I}\right)\right)
        +
        \lVert (\mu_{i})_{i\in I}\rVert_{\bigoplus_{i\in I}^{p}\Psi_{i}} : (\mu_{i})_{i\in I} \in \bigoplus_{i\in I}^{p}\Psi_{i}  \right\}.
    \end{align*}
    Therefore, we have
    \begin{align*}
        (2) = \inf\left\{ \inf C_{\widetilde{(A_{i})_{i\in I}}^{-1}((f_{i})_{i\in I})} : (f_{i})_{i\in I} \in \bigoplus_{i\in I}^{p}\mathcal{B}_{i} \right\}
        = \inf \bigcup_{(f_{i})_{i\in I} \in \bigoplus_{i\in I}^{p}\mathcal{B}_{i}}C_{\widetilde{(A_{i})_{i\in I}}^{-1}((f_{i})_{i\in I})} = (1).
    \end{align*}
    Let $(\mu_{i}^{*})_{i\in I}$ be a minimizer of (1). Then, for each $(\mu_{i})_{i\in I} \in \bigoplus_{i \in I }^{p}\Psi_{i}$, we have
    \begin{align*}
        \mathcal{R}_{L,P}(A(\mu_{i}^{*})_{i\in I})+\lVert (\mu_{i}^{*})_{i\in I}\rVert_{\bigoplus_{i\in I}^{p}\Psi_{i}}
        \le
        \mathcal{R}_{L,P}(A(\mu_{i})_{i\in I})+\lVert (\mu_{i})_{i\in I}\rVert_{\bigoplus_{i\in I}^{p}\Psi_{i}}.
    \end{align*}
    Furthermore, for any $(f_{i})_{i\in I} \in \bigoplus_{i\in I}^{p}\mathcal{B}_{i}$, we obtain:
    \begin{align*}
        \mathcal{R}_{L,P}(A(\mu_{i}^{*})_{i\in I})+\lVert (\mu_{i}^{*})_{i\in I}\rVert_{\bigoplus_{i\in I}^{p}\Psi_{i}}
        &\le
        \inf\left\{
        \mathcal{R}_{L,P}(A(\nu_{i})_{i\in I})+\lVert (\nu_{i})_{i\in I}\rVert_{\bigoplus_{i\in I}^{p}\Psi_{i}}: (\nu_{i})_{i\in I}  \in \widetilde{(A_{i})_{i\in I}}^{-1}((f_{i})_{i\in I}) \right\}
        \\
        &=
        \mathcal{R}_{L,P}(\mathcal{S}(f_{i})_{i\in I})
        +
        \inf\left\{ \lVert (\nu_{i})_{i\in I}\rVert_{\bigoplus_{i\in I }^{p}\Psi_{i}}: (\nu_{i})_{i\in I}  \in \widetilde{(A_{i})_{i\in I}}^{-1}((f_{i})_{i\in I}) \right\}
        \\
        &=
        \mathcal{R}_{L,P}(\mathcal{S}(f_{i})_{i\in I})
        +
        \lVert (f_{i})_{i\in I}\rVert_{\bigoplus_{i\in I}^{p}\mathcal{B}_{i}}
    \end{align*}
    where the last equality follows from \eqref{eq. compatiblity norm comparable}.
    Finally, by the above inequality we conclude that:
    \begin{align*}
        \mathcal{R}_{L,P}(\mathcal{S}(A_{i}\mu_{i}^{*})_{i\in I}) + \lVert (A_{i}\mu_{i}^{*})_{i\in I}\rVert_{\bigoplus_{i\in I}^{p}\mathcal{B}_{i}}
        &\le
        \mathcal{R}_{L,P}(\mathcal{S}(A_{i}\mu_{i}^{*})_{i\in I}) + \lVert (\mu_{i}^{*})_{i\in I}\rVert_{\bigoplus_{i\in I}^{p}\Psi_{i}}
        \\
        &=
        \mathcal{R}_{L,P}(A(\mu_{i}^{*})_{i\in I})+\lVert (\mu_{i}^{*})_{i\in I}\rVert_{\bigoplus_{i\in I}^{p}\Psi_{i}}.
    \end{align*}
\end{proof}
\subsection{Representer theorem of Integral RKBS revisit} 
We aim to understand Theorem 3.9 (the representer theorem for integral RKBSs) from paper \citet{bartolucci2023understanding} using the decomposition methodology introduced in this work. While the previous chapter developed the theory for general L-risk, this chapter focuses exclusively on empirical L-risk. This is because, to the best of our knowledge, except for the case of reflexive Banach spaces considered in paper \citet{combettes2018regularized}, no representer theorem has been developed for general L-risk. For convenience, we restrict our consideration to the following loss function $L:\mathcal{Y}\times \mathbb{R}\rightarrow \mathbb{R}$ instead of the previously defined loss. The following theorem follows directly from Theorem 3.9 in paper \citet{bartolucci2023understanding} and Proposition \ref{prop. multiple kernel method contained in integral RKBS}.

\begin{theorem}[Theorem 3.9 in \cite{bartolucci2023understanding}]
    Assume that the loss function $L(y,\cdot)$ is convex and coercive in the second entry. Then, the problem 
    \begin{align*}
        \inf\left\{\mathcal{R}_{L,D}(f)+\lambda\lVert f\rVert_{\mathcal{F}_{\sigma}(\mathcal{X},\Omega)}: f\in \mathcal{F}_{\sigma}(\mathcal{X},\Omega)\right\}
    \end{align*}
    admits a solution $f_{D,\lambda}$, such that, for all $x\in \mathcal{X}$ $f_{D,\lambda}(x) = \sum_{j=1}^{K}a_{j}\sigma(x,w_{j})$, $a_{j}\in \mathbb{R}\setminus{\{0\}}, w_{j} \in \Omega$.
    \\
    Furthermore, there is a discretized problem
    
    \begin{align*}
        \inf\left\{\mathcal{R}_{L,D}(f)+\lambda\lVert f\rVert_{\mathcal{F}_{\sigma}(\mathcal{X},\Omega)}: f\in \mathcal{F}_{\sigma}(\mathcal{X},\Omega)\right\}
        =
        \underbrace{\inf\left\{\mathcal{R}_{L,D}(f)+\lambda\lVert f\rVert_{\sum_{j\in [K]}\mathcal{L}_{\sigma}^{1}(\delta_{w_{j}})}: f\in \sum_{j\in [K]}\mathcal{L}_{\sigma}^{1}(\delta_{w_{j}})\right\}}_{(*)}
    \end{align*}
    which admits the same minimizer. That is, the problem $(*)$ has a minimizer $f^{*}$, which satisfies $f^{*} = f_{D,\lambda}$.
    \label{thm. Representer theorem of integral RKBS}
\end{theorem}

The following equations demonstrate how the optimization problem in Theorem \ref{thm. Representer theorem of integral RKBS} is transformed into a discretized problem using Proposition \ref{prop. reformulation of optimization problem}, Proposition \ref{prop. reformulate optimization problem of direct sum space}, and Lemma \ref{lem. isometric v.s. have same minimizer}. Furthermore, since we know that $\inf\left\{\mathcal{R}_{L,D}(A\mu) + \lVert \mu\rVert_{\mathcal{M}(\Omega)} : \mu \in \mathcal{M}(\Omega) \right\}$ has $\sum_{j \in [K]}a_{j}\delta_{w_{j}}$ as its solution (see Theorem 3.3 of \cite{bredies2020sparsity} and Theorem 3.9 in \cite{bartolucci2023understanding}), we can observe how the solutions evolve as the initial optimization problem in the measure space (the feature space of integral RKBS) is reformulated into an optimization problem over the finite direct sum of RKBSs.

\begin{align*}
    &\phantom{=}\underbrace{\inf\left\{\mathcal{R}_{L,D}(A\mu) + \lVert \mu\rVert_{\mathcal{M}(\Omega)} : \mu \in \mathcal{M}(\Omega) \right\}}_{\sum_{j \in [K]}a_{j}\delta_{w_{j}}}
    \\
    &=
    \underbrace{\inf\left\{\mathcal{R}_{L,D}(A\Theta(g_{i})_{i\in I}) + \lVert (g_{i})_{i\in I}\rVert_{\bigoplus_{i\in I }^{1}L^{1}(\mu_{i})} : (g_{i})_{i\in I} \in \bigoplus_{i\in I }^{1}L^{1}(\mu_{i})\right\}}_{(a_{1},\dots,a_{K},0,\dots)}
    \\
    &=
    \underbrace{
    \inf\left\{\mathcal{R}_{L,D}(\mathcal{S}(f_{i})_{i\in I}) + \lVert (f_{i})_{i\in I}\rVert_{\bigoplus_{i\in I }^{1}\mathcal{L}^{1}(\mu_{i})} : (f_{i})_{i\in I} \in \bigoplus_{i\in I }^{1}\mathcal{L}^{1}(\mu_{i})\right\}}_{(a_{1}\sigma(\cdot,w_{1}),\dots, a_{K}\sigma(\cdot,w_{K}),0,\dots)}
\end{align*}

\begin{align*}    
    &=
    \underbrace{
    \inf\left\{\mathcal{R}_{L,D}\left(\sum_{j=1}^{K}f_{j}\right) + \lVert (f_{j})_{j\in [K]}\rVert_{\bigoplus_{j\in [K] }^{1}\mathcal{L}^{1}(\delta_{w_{j}})} : (f_{j})_{j\in [K]} \in \bigoplus_{j\in [K] }^{1}\mathcal{L}^{1}(\delta_{w_{j}})\right\}}_{(a_{1}\sigma(\cdot,w_{1}),\dots, a_{K}\sigma(\cdot,w_{K}))}
    \\
    &=\underbrace{
    \inf\left\{\mathcal{R}_{L,D}\left(\sum_{j=1}^{K}f_{j}\right) + \lVert \sum_{j\in [K]} f_{j} \rVert_{\sum_{j\in [K] }^{1}\mathcal{L}^{1}(\delta_{w_{j}})} : \sum_{j\in [K]} f_{j} \in \sum_{j\in [K] }\mathcal{L}^{1}(\delta_{w_{j}})\right\}}_{\sum_{j\in [K]}a_{j}\sigma(\cdot,w_{j})}
\end{align*}

\section{Conclusion and Future Work}
We showed that there is a compatibility property between the direct sum of feature spaces and the sum of RKBSs. By using this, we can decompose a class of integral RKBS $\mathcal{F}_{\sigma}(\mathcal{X},\Omega)$ into the sum of p-norm RKBSs $\{\mathcal{L}_{\sigma}^{1}(\mu_{i})\}_{i\in I}$. The advantage of this analytical method is that it allows for a more structural understanding of the RKBS class through an appropriate decomposition approach. In Section 5, we partially explained these advantages by comparing the integral RKBS class with the previously known sum of RKHSs. We also discussed, from an optimization perspective, how an optimization problem in the hypothesis space can be reformulated as an optimization problem in the feature space including the cases where the hypothesis space is given by a direct sum or a sum of RKBSs. 
Additionally, through these insights, we expect that it would be helpful in designing multiple kernel learning algorithms for the RKBS class.

\section*{Acknowledgements}
This work is supported by the National Research Foundation of Korea(NRF) grant funded by the Korea government(MSIT) (No.RS-2024-00421203, 2021R1A2C3010887).
\bibliographystyle{plainnat}
\bibliography{references}

\end{document}